\RequirePackage{fix-cm}
\documentclass[smallextended]{amsart}

\usepackage{mathptmx}

\usepackage[url=false, isbn=false, doi=false, backend=bibtex, firstinits=true, backref=true]{biblatex}
\usepackage{amssymb,amsmath,amsfonts,amsthm}
\usepackage{graphicx}
\usepackage[mathscr]{eucal}
\usepackage{multicol}
\usepackage{bm}

\allowdisplaybreaks

\usepackage[pdfpagelabels,hypertexnames=true,plainpages=false,naturalnames=false]{hyperref}

\usepackage{color}
\definecolor{darkblue}{rgb}{0,0.1,0.5}
\hypersetup{colorlinks,linkcolor=darkblue,anchorcolor=darkblue,citecolor=darkblue}

\newtheorem{theorem}{Theorem}[section]
\newtheorem{lemma}[theorem]{Lemma}
\newtheorem{proposition}[theorem]{Proposition}
\newtheorem{corollary}{Corollary}

\newtheorem{definition}{Definition}[section]

\newtheorem{remark}{Remark}

\numberwithin{equation}{section}

\newcommand{\zero}{\boldsymbol{0}}
\newcommand{\Fcal}{{\mathcal F }}
\newcommand{\Mcal}{{\mathcal M }}
 
\newcommand{\Mscr}{\mathscr{M}}
\newcommand{\Rscr} {\mathscr{R}}

\newcommand{\dist}{{\operatorname {dist}}}

\newcommand{\Symb}{\mathcal{A}}
\newcommand{\SymbWR}{\mathcal{B}_{q}}
\newcommand{\N}{\mathbb{N}}
\newcommand{\Z}{\mathbb{Z}}
\newcommand{\R}{\mathbb{R}}
\newcommand{\Prob}{\mathbb{P}}
\newcommand{\Press}{{\mathrm P}}
\newcommand{\gammatwo}{468}

\newcommand{\block}{\mathrm{B}}
\newcommand{\supp}{\mathrm{supp}}
\newcommand{\PFunc}{\mathrm{Z}^\Phi_{\block_n}}
\newcommand{\EFunc}{\mathrm{E}^\Phi}
\newcommand{\past}{\mathcal{P}}
\newcommand{\Path}{\mathrm{T}}
\newcommand{\Circ}{\mathrm{C}}

\newcommand{\tBdry}{\partial_{\uparrow}}
\newcommand{\bBdry}{\partial_{\downarrow}}
\newcommand{\veczero}{\boldsymbol{0}}
\newcommand{\vecone}{\boldsymbol{1}}

\begin{document}

\title[Representation and poly-time approximation for pressure of $\Z^2$ lattice models]{Representation and poly-time approximation for pressure of $\Z^2$ lattice models in the non-uniqueness region}

\author{Stefan Adams}
\address{Stefan Adams\\Mathematics Institute\\University of Warwick\\Zeeman Building\\Coventry CV4 7AL\\United Kingdom}
\email{s.adams@warwick.ac.uk}

\author{Raimundo Brice\~no}
\address{Raimundo Brice\~no\\Department of Mathematics\\The University of British Columbia\\1984 Mathematics Road\\Vancouver, B.C., V6T 1Z2\\Canada}
\email{raimundo@math.ubc.ca}

\author{Brian Marcus}
\address{Brian Marcus\\Department of Mathematics\\The University of British Columbia\\1984 Mathematics Road\\Vancouver, B.C., V6T 1Z2\\Canada}
\email{marcus@math.ubc.ca}

\author{Ronnie Pavlov}
\address{Ronnie Pavlov\\Department of Mathematics\\University of Denver\\2280 S. Vine St.\\Denver, CO 80208\\USA}
\email{rpavlov@du.edu}


\maketitle

\begin{abstract}
We develop a new pressure representation theorem for nearest-neigh\-bour Gibbs interactions and apply this to obtain the existence of efficient algorithms for approximating the pressure in the $2$-dimensional ferromagnetic Potts, multi-type Wi\-dom-Rowlinson and hard-core models. For Potts, our results apply to every inverse temperature but the critical. For Widom-Rowlinson and hard-core, they apply to certain subsets of both the subcritical and supercritical regions. The main novelty of our work is in the latter. 
\end{abstract}

\tableofcontents


\section{Introduction}
\label{section1}

The \emph{pressure} of an interaction is a crucial quantity studied in statistical mechanics and dynamical systems.  In the former, it coincides with the \emph{specific Gibbs free energy} of a statistical mechanical system (e.g. \cite[Part III]{1-georgii} and \cite[Chapter 3-4]{1-ruelle}). In the latter, it is a generalization of \emph{topological entropy} and has many applications in a wide variety of classes of dynamical systems, ranging from symbolic to smooth systems (e.g. \cite{1-bowen,1-keller,1-walters}).

In this paper, we continue the development in \cite{1-gamarnik,1-marcus,1-briceno} of representing pressure with a simplified expression and using this to prove the existence of efficient algorithms for approximating pressure.

We consider \emph{nearest-neighbour (n.n.) real-valued interactions} $\Phi$ on $\Z^d$, i.e. interactions defined only on configurations on single sites and pairs of adjacent sites. Since pressure is normally defined for stationary interactions, we assume that our interactions are stationary here. Also, we allow the possibility of forbidden configurations $\mathcal{E}$ on pairs of adjacent sites, and so the space of feasible configurations on $\Z^d$ may be constrained. In the dynamical systems literature, the space of such feasible configurations is known as a \emph{nearest-neighbour shift of finite type (n.n. SFT)}, that here we denote $\Omega(\mathcal{E})$ (see Section \ref{specification}).

A \emph{specification} $\pi$ for a n.n. interaction $\Phi$ is a uniquely determined collection of Borel probability measures $\pi_\Lambda^\xi$ given in an explicit form in terms of $\Phi$, for configurations on finite subsets $\Lambda$ of $\Z^d$ and feasible configurations $\xi$ on the boundary of $\Lambda$. A \emph{Gibbs measure} $\mu$ for a n.n. interaction $\Phi$ is a Borel probability measure on $\Omega(\mathcal{E})$, whose conditional probability distributions on any such $\Lambda$ agree with the specification for $\Phi$ for all boundary conditions $\xi$ of positive $\mu$-measure.

Gibbs measures exist for all n.n. interactions (and, indeed, for much more general interactions), but a given n.n. interaction may have more than one Gibbs measure. In many cases, including the ones of most interest to us here, there is a n.n. interaction $\Phi$ which gives rise to a parameterized family of interactions $\left\{\zeta \Phi\right\}_{\zeta > 0}$, and uniqueness of Gibbs measures holds for sufficiently small $\zeta$ (the so-called \emph{subcritical region}) and uniqueness fails for  sufficiently large $\zeta$ (the so-called \emph{supercritical region}).

Given a n.n. interaction $\Phi$ on a n.n. SFT $\Omega(\mathcal{E})$, we can associate an \emph{energy} to any feasible configuration on a finite subset $\Lambda$ of $\Z^d$. The \emph{partition function} $\mathrm{Z}^\Phi_\Lambda$ of $\Phi$ on $\Lambda$ corresponds to the sum over all feasible configurations on $\Lambda$ of a function (namely, $e^{-x}$) of their corresponding energy, and the pressure $\Press(\Phi)$ is defined as the asymptotic exponential growth rate of the partition function $\PFunc$ on an increasing sequence of boxes $\block_n$ which exhausts $\Z^d$, as $n \rightarrow \infty$. Note that $\Press(\Phi)$ implicitly depends on $\Omega(\mathcal{E})$. 

When $d=1$, there is a closed-form expression for $\Press(\Phi)$ in terms of the largest eigenvalue of an adjacency matrix formed from $\Phi$ (see \cite[p. 99]{1-krieger}). In contrast, when $d \geq 2$, there are very few n.n. interactions $\Phi$ for which $\Press(\Phi)$ is known exactly.
 
There is much work in the literature on numerical approximations of $\Press(\Phi)$, both for somewhat general $\Phi$ and somewhat specific $\Phi$ (see \cite{3-baxter,1-friedland}). In our paper, we take a theoretical computer science point of view (see \cite{1-ko}): an algorithm for computing a real number $r$ is said to be \emph{poly-time} if for every $N \in \N$, the algorithm outputs an approximation $r_N$ to $r$, which is guaranteed to be accurate within $\frac{1}{N}$ and takes time at most polynomial in $N$ to compute. In that case, we say that $r$ is \emph{poly-time computable}.

One of our goals is to prove the existence of poly-time algorithms for $\Press(\Phi)$ under certain assumptions on $\Phi$ and $\Omega(\mathcal{E})$. While one might expect such algorithms to exist for most $\Phi$ and $\Omega(\mathcal{E})$ of practical interest, there exist $\Omega(\mathcal{E})$ for which even $\Press(0)$ (which corresponds to the topological entropy of $\Omega(\mathcal{E})$, when the n.n. interaction is $\Phi \equiv 0$) is not poly-time computable and some for which $\Press(0)$ is not computable at any rate (see \cite{1-hochman}). However, the closed-form expression when $d = 1$ mentioned above, always gives a poly-time algorithm in that case.  

We follow an approach initiated by Gamarnik and Katz \cite{1-gamarnik}, and further developed by two of the authors \cite{1-marcus} of the present paper. The basic idea is motivated by the \emph{variational principle} \cite[Section 4.4]{1-keller}, which asserts that $\Press(\Phi)$ is the supremum over all stationary Borel probability measures $\mu$ on $\Omega(\mathcal{E})$ of the sum of two quantities: one quantity is the \emph{measure-theoretic entropy} $h(\mu)$ of $\mu$ and the other quantity is the integral, with respect to $\mu$, of a simple explicit function $A_\Phi: \Omega(\mathcal{E}) \to \R$, determined by $\Phi$. The entropy $h(\mu)$ can be expressed as the integral, also with respect to $\mu$, of a function known as the \emph{information function} $I_\mu$, i.e. $h(\mu) = \int{I_\mu}d\mu$. The supremum is always achieved by a Gibbs measure $\mu$ for $\Phi$, and so for such $\mu$, we can write $\Press(\Phi) = \int{(I_\mu  + A_\Phi)}d\mu$.

The idea of \cite{1-gamarnik} was to represent $\Press(\Phi)$ as the integral of the same integrand, but with respect to a simpler measure $\nu$, i.e. $\Press(\Phi) = \int{(I_\mu  + A_\Phi)}d\nu$. This is what we call a \emph{pressure representation} and requires some assumptions on $\mu$, $\nu$ and $\Omega(\mathcal{E})$.

A pressure representation becomes especially useful for approximating $\Press(\Phi)$ in the case that $\nu$ is a periodic point measure, i.e. a measure which assigns equal weight to each distinct translation of a given periodic configuration (this was the only case considered in \cite{1-gamarnik}). Then $\int{(I_\mu  + A_\Phi)}d\nu$ becomes a finite sum. The terms in this sum corresponding to $A_\Phi$ are easy to compute. In this way, the problem of approximating $\Press(\Phi)$ (and therefore proving that $\Press(\Phi)$ is poly-time computable) reduces to approximating $I_\mu$ on a single periodic configuration and its translates.

The pressure representation theorems in \cite{1-gamarnik} and \cite{1-marcus}, as well as in our paper (see Theorem \ref{press-rep}), work in all dimensions $d$. Among other conditions, these results require conditions on $\Omega(\mathcal{E})$ and a convergence condition for certain sequences of finite volume half-plane measures (different convergence conditions in the different results). In the case $d=2$, if the convergence holds at exponential rate, then one obtains a poly-time algorithm for approximating $\Press(\Phi)$ (see Theorem \ref{press-aprox}). For $d > 2$, one can deduce an algorithm for approximating $\Press(\Phi)$ with sub-exponential but not polynomial rate.

In \cite{1-gamarnik} and \cite{1-marcus}, the convergence condition is given in terms of the information function $I_\mu$ of a stationary Gibbs measure $\mu$ for the interaction. In our paper, the condition is given in terms of a closely related function $\hat{I}_\pi$, which depends only on the specification $\pi$ of the interaction (see Section \ref{I-hat}), in contrast with \cite{1-gamarnik} and \cite{1-marcus}. This is natural, since the pressure depends only on the interaction and not on any particular Gibbs measure $\mu$.

In \cite{1-gamarnik}, the convergence condition is \emph{strong spatial mixing} of a Gibbs measure $\mu$ for the n.n. interaction $\Phi$. This condition is known to imply that there is a unique Gibbs measure for $\Phi$ and thus can be applied only in the uniqueness (subcritical) region of a given model. The convergence conditions in \cite{1-marcus} are weaker but also apply primarily to this region. However, in our paper, since our convergence condition depends only on the interaction, one might expect that the pressure representation and approximation results can apply in the non-uniqueness region as well. Indeed, they do.  As illustrations, we apply these results to explicit subcritical and supercritical sub-regions of the $2$-dimensional \emph{(ferromagnetic) Potts}, \emph{(multi-type) Widom-Rowlinson} and \emph{hard-core} models. In particular, for the pressure approximation results for these models, we establish the required exponential convergence conditions. However, we believe that our results are applicable to a much broader class of models, in particular satisfying weaker conditions on $\Omega(\mathcal{E})$ (e.g. the \emph{topological strong spatial mixing property}, introduced in \cite{1-briceno}). We remark that the strong spatial mixing condition of \cite{1-gamarnik} is a much stronger version of our condition, and so in this sense our results generalize some results of that paper (in particular, for the hard-core model on $\Z^2$).

In the case of the $2$-dimensional ferromagnetic Potts model, we obtain a pressure representation and efficient pressure approximation for all $\beta \neq \beta_{\rm c}(q)$, where $q$ is the number of colours, $\beta$ is the inverse temperature and $\beta_{\rm c}(q) = \log(1 + \sqrt{q})$ is the \emph{critical value} which separates the uniqueness and non-uniqueness regions. Our proof in the non-uniqueness region generalizes a result from \cite{1-chayes} for $q=2$ (i.e. the \emph{Ising model}) and we closely follow their proof, which relies heavily on a coupling with the \emph{bond random-cluster model} and planar duality. For the uniqueness region, our result follows from \cite{1-alexander}. (See Corollary \ref{potts-wr-hard}, part 1.)

For the Widom-Rowlinson and hard-core models, our results are not as complete as in the Potts case, since the subcritical and supercritical regions for these two models haven't been completely determined, in contrast with the Potts model. We also expect our results can be improved, because they only apply to proper subsets of the currently known uniqueness/non-uniqueness regions.

For the Widom-Rowlinson model, in the supercritical region, we use a variation of the disagreement percolation technique introduced in \cite{1-berg}, combined with the connection between the Widom-Rowlinson model and the \emph{site random-cluster model}. In the subcritical region, we apply directly the results in \cite{1-berg}. (See Corollary \ref{potts-wr-hard}, part 2.)

For the hard-core model, in the supercritical region, we combine the coupling in \cite{1-berg} and a Peierls argument used by Dobrushin (see \cite{1-dobrushin}). In the subcritical region, we use a recent result on strong spatial mixing for the hard-core model in $\Z^2$. (See Corollary \ref{potts-wr-hard}, part 3.)

For the Potts model, we also extend the pressure representation, by a continuity argument, to give an expression for the pressure at criticality. It is of interest that there is an exact, explicit, but non-rigorous, formula for the pressure at criticality due to Baxter \cite{1-baxter}. So, our rigorously obtained expression should agree with that formula, though we do not know how to prove this statement. It seems that Baxter's explicit expression gives a poly-time approximation algorithm, but we cannot justify that our expression is poly-time computable.

We remark that the finite volume half-plane measures mentioned above typically are constant on their bottom boundaries and thus are related to \emph{wetting models} (see \cite{1-pfister,2-russo}). Our proofs are related with such models where the interaction with the hard-wall is the same as the bulk interaction. 

The remainder of the paper is organized as follows. Since we have drawn heavily on many concepts from many different sources, for the convenience of the reader we have collected a good deal of relevant background material early in the paper. This can be found in Section \ref{section2}, Section \ref{section3}, Section \ref{section4}, Section \ref{section5} and Section \ref{section7}, with the notable exception of  Lemma \ref{pseudoMRF} in Section \ref{section5}, there is very little new material in those sections. In Section \ref{section2} and Section \ref{section3}, we review the fundamentals on configuration spaces on $\Z^d$, Gibbs measures and pressure. In Section \ref{section4}, we review the specific lattice spin systems models to which we apply our main results, and in Section \ref{section5} we review the bond and site random-cluster models which are intimately connected with two of our models. Our pressure representation theorem is contained in Section \ref{section6}. We review spatial mixing and stochastic dominance in Section \ref{section7} and use these concepts in Section \ref{section8} to help establish exponential convergence results for our models. Finally, in Section \ref{section9}, we combine our pressure representation theorem and our exponential convergence results in Section \ref{section8} to obtain pressure representations and poly-time algorithms for our models.


\section{Definitions and preliminaries}
\label{section2}

\subsection{Hypercubic lattice $\Z^d$}
Given $d \in \N$, we consider the \emph{$d$-dimensional hypercubic lattice} $\Z^d$, which can be regarded as a countable graph with regular degree $2d$, where $V(\Z^d) = \Z^d$ is the set of sites and $E(\Z^d) = \left\{\{x,y\}: x,y \in \Z^d, \|x-y\|=1\right\}$ is the set of bonds, with $\|x\| = \sum_{i=1}^{d}\left|x_i\right|$ the $1$-norm. We will mainly focus our attention on the case $d =2$.

Two sites $x,y \in \Z^d$ are \emph{adjacent} if $\{x,y\} \in E(\Z^d)$ and we will denote this by $x \sim y$. All subsets of sites in $\Z^d$ will be denoted with uppercase Greek letters (e.g. $\Lambda$, $\Delta$, $\Theta$, etc.). Whenever a finite set $\Delta$ is contained in an infinite set $\Lambda$, we denote this by $\Delta \Subset \Lambda$. The \emph{(outer) boundary} of $\Lambda \subseteq \Z^d$ is the set $\partial \Lambda$ of $x \in \Z^d \setminus \Lambda$ which are adjacent to some element of $\Lambda$, i.e. $\partial \Lambda := \left\{x \in \Lambda^c: \dist(\{x\},\Lambda) = 1\right\}$, where $\dist(\Lambda_1,\Lambda_2) = \min_{x \in \Lambda_1, y \in \Lambda_2} \|x - y\|$, for $\Lambda_1, \Lambda_2 \subseteq \Z^d$. We also write the \emph{closure of $\Lambda$} as $\overline{\Lambda} := \Lambda \cup \partial \Lambda$. On the other hand, the \emph{inner boundary} of $\Lambda \subseteq \Z^d$ is the set $\underline{\partial} \Lambda := \partial \Lambda^c$ of $x \in \Lambda$ which are adjacent to some element of $\Lambda^c$. When denoting subsets of $\Z^d$ that are singletons, brackets will be usually omitted, e.g. $\dist(x,\Lambda)$ will be regarded to be the same as $\dist(\{x\},\Lambda)$.

A \emph{path} $\Path \Subset \Z^d$ will be any sequence of distinct sites $x_1,\dots,x_n$ such that $x_i \sim x_{i+1}$, for all $1 \leq i < n$. Similarly, a \emph{circuit} $\Circ \Subset \Z^d$ will be any path $x_1,\dots,x_n$ with $n \geq 4$ such that, in addition, $x_n \sim x_1$. We will say that the circuit is \emph{simple} if $x_i \sim x_j$ iff $|i-j|= 1$ or $\{i,j\} = \{1,n\}$ (in particular, $x_1,\dots, x_n$ are all distinct). For $\Delta,\Theta \subseteq \Z^d$, a \emph{path from $\Delta$ to $\Theta$} is a path $\Path$ whose first site is in $\Delta$ and whose last site is in $\Theta$. A set $\Lambda \subseteq \Z^d$ is said to be \emph{connected} if for every $x,y \in \Lambda$, there is a path $\Path$ from $x$ to $y$ contained in $\Lambda$ (i.e. $\Path \subseteq \Lambda$). A set $\Lambda \Subset \Z^2$ is said to be \emph{simply lattice-connected} if $\Lambda$ and $\Lambda^c$ are both connected.

In $\Z^d$ we can also define an alternative notion of adjacency and therefore,  an alternative notion of boundary, inner boundary, closure, path, connectedness, etc., by replacing the $1$-norm $\| \cdot \|$ with the $\infty$-norm $\| \cdot \|_\infty$, defined as $\|x\|_\infty = \max_{i=1,\dots,d}\left|x_i\right|$, for $x \in \Z^d$. When referring to these notions with respect to the $\infty$-norm, we will always add a $\star$ superscript and talk about $\star$-adjacency $x \overset{\star}{\sim} y$ , $\star$-boundary $\partial^\star \Lambda$, inner $\star$-boundary $\underline{\partial}^\star \Lambda$, $\star$-closure $\overline{\Lambda}^\star$, $\star$-path, $\star$-connectedness, etc. Notice that two sites $x$ and $y$ are $\star$-adjacent if they are adjacent in a version of the $d$-dimensional hypercubic lattice $\Z^d$ including in addition diagonal bonds. We will denote this version of the lattice by $\Z^{d,\star}$.

A natural order on $\Z^d$ is the so-called \emph{lexicographic order}, where $y \prec x$ (or $x \succ y$) if and only if $y \neq x$ and, for the smallest $i$ for which $y_i \neq x_i$, $y_i$ is strictly smaller than $x_i$. We also denote $y \preccurlyeq x$ (or $x \succcurlyeq y$) if $y \prec x$ or $y =  x$. Considering this order, we define the family of sets $S_{y,z} \Subset \Z^d$ as:
\begin{equation}
S_{y,z} := \left\{x \succcurlyeq \veczero: -y \leq x \leq z\right\},
\end{equation}
where $y,z \in \Z^d$ are such that $y,z \geq \veczero$ (here $\veczero$ denotes the vector $(0,\dots,0) \in \Z^d$ and $\geq$, the coordinate-wise comparison of vectors). In addition, given $n \in \N$, we define the \emph{$n$-block} as the set $\block_n := [-n,n]^d \cap \Z^d$ and we abbreviate by $S_n$ the set $S_{\vecone n,\vecone n} = \block_n \setminus \past$, where $\past := \left\{x \in \Z^d: x \prec \veczero\right\}$ denotes the \emph{(lexicographic) past} of $\Z^d$ and $\vecone$, the vector $(1,\dots,1) \in \Z^d$.

\subsection{Configuration spaces}

Consider a finite set of \emph{symbols} $\Symb$ called the \emph{alphabet}. A \emph{configuration} is a map $\theta: \Lambda \to \Symb$, for some $\emptyset \neq \Lambda \subseteq \Z^d$ (i.e. $\theta \in \Symb^\Lambda$), which will be usually denoted with lowercase Greek letters $\theta$, $\tau$, $\upsilon$. The set $\Lambda$ is called the \emph{shape} of $\theta$, and a configuration will be said to be finite if its shape is finite. For any configuration $\theta$ with shape $\Lambda$ and $\Delta \subseteq \Lambda$, $\theta(\Delta)$ denotes the restriction of $\theta$ to $\Delta$, i.e. the \emph{sub-configuration} of $\theta$ occupying $\Delta$. We will usually save the Greek letters $\xi$ and $\eta$ to denote configurations whose shape is the boundary $\partial \Lambda$ of some given set $\Lambda$. For $\Lambda_1$ and $\Lambda_2$ disjoint sets, $\theta \in \Symb^{\Lambda_1}$ and $\tau \in \Symb^{\Lambda_2}$, $\theta\tau$ will be the configuration on $\Lambda_1 \sqcup \Lambda_2$ defined by $(\theta\tau)(\Lambda_1) = \theta$ and $(\theta\tau)(\Lambda_2) = \tau$. For $a \in \Symb$ and $\Lambda \subseteq \Z^d$, $a^\Lambda$ denotes the configuration of all $a$'s on $\Lambda$. A \emph{point} is a configuration with shape $\Z^d$, i.e. an element of $\Symb^{\Z^d}$, usually denoted with the Greek letter $\omega$.

Given sets $\Lambda_1, \Lambda_2 \subseteq \Z^d$, $\Delta \subseteq \Lambda_1 \cap \Lambda_2$ and a pair of configurations $\theta \in \Symb^{\Lambda_1}$, $\tau \in \Symb^{\Lambda_2}$, we define the \emph{set of  $\Delta$-disagreement} as:
\begin{equation}
 \Sigma_\Delta(\theta,\tau) := \left\{x \in \Delta: \theta(x) \neq \tau(x)\right\},
\end{equation}
i.e. the set of sites in $\Delta$ where $\theta$ and $\eta$ differ.

The map $\sigma:\Z^d \times \Symb^{\Z^d} \to \Symb^{\Z^d}$ will be the \emph{shift action} on $\Symb^{\Z^d}$ defined by $(x,\omega) \mapsto \sigma_x(\omega)$, where $x \in \Z^d$ and $\omega \in \Symb^{\Z^d}$, with $\left(\sigma_x(\omega)\right)(y) = \omega(x+y)$, for $y \in \Z^d$. We also extend the shift action $\sigma_x$ to configurations with arbitrary shapes, i.e. given $\theta \in \Symb^\Lambda$, we define $\sigma_x(\theta) \in \Symb^{\Lambda-x}$ as the configuration such that $\left(\sigma_x(\theta)\right)(y) = \theta(x+y)$, for $y \in \Lambda-x$.

Given a point $\omega \in \Symb^{\Z^d}$, we define its \emph{orbit} as the set $\mathrm{O}(\omega) := \left\{\sigma_x(\omega)\right\}_{x \in \Z^d}$. We will say that a point $\omega$ is \emph{periodic} if $|\mathrm{O}(\omega)| < \infty$.

\subsection{Borel probability measures}

Given a configuration $\theta \in \Symb^\Lambda$, we define the \emph{cylinder set} $[\theta]_\Lambda := \{\omega \in \Symb^{\Z^d}: \omega(\Lambda) = \theta\}$ (or just $[\theta]$, if $\Lambda$ is understood). We denote by $\mathcal{F}_\Lambda$ the $\sigma$-algebra generated by all the cylinder sets with shape $\Lambda$ and set $\mathcal{F}: = \mathcal{F}_{\Z^d}$. 

A \emph{Borel probability measure} $\mu$ on $\mathcal{F}$ is a measure determined by its values on cylinder sets of finite configurations such that $\mu(\Symb^{\Z^d}) = 1$. Given a cylinder set $[\theta]$, we will just write $\mu(\theta)$ for the value of $\mu([\theta])$. The \emph{support} of such a measure $\mu$ is defined as:
\begin{equation}
\supp(\mu) := \left\{\omega \in \Symb^{\Z^d}: \mu(\omega(\Lambda)) > 0, \mbox{ for all } \Lambda \Subset \Z^d\right\}.
\end{equation}

Given $\Delta \subseteq \Lambda \subseteq \Z^d$ and a measure $\mu$ on $\mathcal{F}_\Lambda$, we denote by $\left.\mu\right|_\Delta$ the restriction (or projection or marginalization) of $\mu$ to $\mathcal{F}_\Delta$.

A measure $\mu$ is \emph{shift-invariant} (or \emph{stationary}) if $\mu(\sigma_{x}(A)) = \mu(A)$, for all measurable sets $A \in \mathcal{F}$ and $x \in \Z^d$. Given any point $\omega \in \Symb^{\Z^d}$ and $A \in \mathcal{F}$, we define the \emph{delta-measure supported on $\omega$} as the measure:
\begin{equation}
\delta_\omega(A) =
\begin{cases}
1	&	\mbox{ if } \omega \in A,	\\
0	&	\mbox{ otherwise.}
\end{cases}
\end{equation}

If $\omega$ is a periodic point with orbit $\mathrm{O}(\omega) = \left\{\omega_1,\dots,\omega_k\right\}$, we define $\nu^\omega$ to be the shift-invariant Borel probability measure supported on $\mathrm{O}(\omega)$ given by:
\begin{equation}
\nu^\omega := \frac{1}{k}\left(\delta_{\omega_1} + \cdots + \delta_{\omega_k} \right).
\end{equation}

\subsection{Markov random fields}

\begin{definition}
Given $\Lambda \subseteq \Z^d$, a probability measure $\rho$ on $\Symb^{\Lambda}$ is a \emph{Markov random field ($\Lambda$-MRF)} if, for any subset $\Theta \Subset \Lambda$, any $\theta \in \Symb^\Theta$, any $\Delta \Subset \Lambda$ s.t. $\partial \Theta \cap \Lambda \subseteq \Delta \subseteq \Lambda \setminus \Theta$, and any $\tau \in \Symb^\Delta$ with $\rho(\tau) > 0$, it is the case that:
\begin{equation}
\rho\left(\theta \middle\vert \tau\right) = \rho\left(\theta  \middle\vert \tau(\partial \Theta \cap \Lambda)\right).
\end{equation}

In other words, an MRF is a measure where every finite configuration conditioned to its boundary is independent of the configuration on the complement.
\end{definition}


\section{Specifications, Gibbs measures and pressure}
\label{section3}

\subsection{Gibbs specifications}
\label{specification}
Fix a dimension $d \in \N$ and let $\mathcal{E} = (\mathcal{E}_1,\dots,\mathcal{E}_d)$ be a \emph{set of constraints} such that $\mathcal{E}_i \subseteq \Symb^2$, for $i = 1,\dots,d$. Given any set $\Lambda \subseteq \Z^d$ and a configuration $\theta \in \Symb^\Lambda$, we say that $\theta$ is \emph{feasible} for $\mathcal{E}$ if for every $x \in \Lambda$ such that $\{x,x+e_i\} \subseteq \Lambda$, we have that $(\theta(x), \theta(x+e_i)) \notin \mathcal{E}_i$, where $e_1,\dots,e_d$ is the canonical basis. The \emph{nearest-neighbour shift of finite type (n.n. SFT)} $\Omega(\mathcal{E})$ induced by $\mathcal{E}$, is the set of points:
\begin{equation}
\Omega(\mathcal{E}) := \left\{\omega \in \Symb^{\Z^d}: \omega \mathrm{~is~feasible}\right\}.
\end{equation}

We will always assume that $\Omega(\mathcal{E}) \neq \emptyset$.

In the symbolic dynamics literature, a feasible configuration on a set $\Lambda$ is called \emph{locally admissible}, and is called \emph{globally admissible} if it also extends to a point of $\Omega(\mathcal{E})$.

Notice that $\Omega(\mathcal{E})$ is always a shift-invariant set, i.e. $\sigma_x(\Omega(\mathcal{E})) = \Omega(\mathcal{E})$, for all $x \in \Z^d$. Given a n.n. SFT $\Omega(\mathcal{E})$, $\mathcal{M}_{1}(\Omega(\mathcal{E}))$ denotes the set of Borel probability measures whose support $\supp(\mu)$ is contained in $\Omega(\mathcal{E})$ and $\mathcal{M}_{1,\sigma}(\Omega(\mathcal{E})) \subseteq \mathcal{M}_{1}(\Omega(\mathcal{E}))$, the corresponding subset of shift-invariant Borel probability measures. Given a configuration $\theta \in \Symb^\Lambda$, $[\theta]^{\Omega(\mathcal{E})}_\Lambda$ will denote the set $[\theta]_\Lambda \cap \Omega(\mathcal{E})$ (or just $[\theta]^{\Omega(\mathcal{E})}$ if $\Lambda$ is understood).

\begin{definition}
A \emph{nearest-neighbour (n.n.) interaction} for a set of constraints $\mathcal{E}$ is a real-valued shift-invariant function $\Phi$ from the set of configurations on sites $x$ and feasible configurations on bonds $\{x,x+e_i\}$ to $\R$, for $x \in \Z^d$ and $i = 1,\dots,d$. Here, shift-invariance means that $\Phi(\sigma_{x}(\theta)) = \Phi(\theta)$ for configurations $\theta$ on sites and bonds, and for all $x \in \Z^d$.
\end{definition}

Often in the literature a n.n. interaction is not required to be shift-invariant. Our assumption of shift-invariance on a n.n. interaction fits naturally with the shift-in\-va\-rian\-ce of a n.n. SFT. Clearly, a n.n. interaction is defined by only finitely many numbers, namely the values of the interaction on configurations on $\{\veczero\}$ and bonds $\{\veczero,e_i\}$, for $i = 1,\dots,d$.

We can view an interaction $\Phi$ as implicitly determining the constraints $\mathcal{E}$, and hence $\Omega(\mathcal{E})$, by the absence of $\mathcal{E}$ from the domain of $\Phi$. Some authors incorporate the constraints by allowing the interaction to take the value $+\infty$.

\begin{definition}
Given a n.n. interaction $\Phi$ for a set of constraints $\mathcal{E}$ and a set $\Lambda \Subset \Z^d$, we define the \emph{energy function} $\EFunc_{\Lambda}: \Symb^\Lambda \to \R$ as:
\begin{equation}
\EFunc_{\Lambda}(\theta) := \sum_{x \in \Lambda} \Phi(\theta(x)) + \sum_{i=1}^d \sum_{\{x,x+e_i\} \subseteq \Lambda} \Phi(\theta(\{x,x+e_i\})),
\end{equation}
where $\theta$ is any feasible configuration in $\Symb^\Lambda$. We define the \emph{partition function} of $\Lambda$ as:
\begin{equation}
\mathrm{Z}_{\Lambda}^\Phi := \sum_{\theta~\mathrm{feasible}} \exp\left(-\EFunc_\Lambda(\theta)\right),
\end{equation}
and the following \emph{boundary-free} probability measure on $\Symb^\Lambda$:
\begin{equation}
\pi^{(f)}_\Lambda(\theta) :=
\begin{cases}
\frac{1}{Z^{\Phi}_\Lambda}\exp\left(-\EFunc_\Lambda(\theta)\right)	&	\mathrm{if}~\theta~\mathrm{is~feasible},	\\
0													&	\mathrm{otherwise.}
\end{cases}
\end{equation}

Analogously, for an arbitrary $\omega \in \Omega(\mathcal{E})$, we can take $\xi = \omega(\partial \Lambda)$ and consider:
\begin{equation}
\mathrm{Z}^{\Phi,\xi}_{\Lambda} := \sum_{\theta:~\theta\xi~\mathrm{feasible}} \exp\left(-\EFunc_{\overline{\Lambda}}(\theta\xi)\right),
\end{equation}
and then define the \emph{$\xi$-boundary} probability measure on $\Symb^\Lambda$:
\begin{equation}
\pi^{\xi}_\Lambda(\theta) :=
\begin{cases}
\frac{1}{\mathrm{Z}^{\Phi,\xi}_\Lambda}\exp\left(-\EFunc_{\overline{\Lambda}}(\theta\xi)\right)	&	\mathrm{if}~\theta\xi~\mathrm{is~feasible},	\\
0																		&	\mathrm{otherwise.}
\end{cases}
\end{equation}

The collection $\pi = \{\pi^{\xi}_\Lambda\}_{\Lambda,\xi}$ is called a \emph{$\Z^d$ Gibbs specification} for the n.n. interaction $\Phi$. For $\Delta \subseteq \Lambda$ and $\tau \in \Symb^{\Delta}$, we marginalize as follows:
\begin{equation}
\pi^\xi_\Lambda(\tau) = \sum_{\theta \in \Symb^{\Lambda}: \theta(\Delta) = \tau} \pi_\Lambda^{\xi}(\theta).
\end{equation}
\end{definition}

Notice that each $\pi^{\xi}_\Lambda$ is an MRF on $\Symb^\Lambda$. In addition, a Gibbs specification $\pi$ as defined above is always stationary, in the sense that $\pi^{\sigma_x(\xi)}_{\Lambda-x}(\sigma_x(A)) = \pi^{\xi}_\Lambda(A)$, for every $A \subseteq \Symb^\Lambda$. We will usually think of the set of restrictions $\mathcal{E}$ implicit when considering a n.n. interaction $\Phi$. Given a point $\omega \in \Omega(\mathcal{E})$, we will abbreviate:
\begin{equation}
\pi_\Lambda^\omega(\cdot) := \pi_\Lambda^{\omega(\partial \Lambda)}(\cdot).
\end{equation}

\subsection{Gibbs measures}

\begin{definition}
A \emph{nearest-neighbour (n.n.) Gibbs measure} for a n.n. interaction $\Phi$ is a measure $\mu \in \mathcal{M}_{1}(\Omega(\mathcal{E}))$ such that for any $\Lambda \Subset \Z^d$ and $\omega \in \Symb^{\Z^d}$ with $\mu(\omega(\partial \Lambda)) > 0$, we have that $\mathrm{Z}^{\Phi,\omega(\partial \Lambda)}_{\Lambda} > 0$ and:
\begin{equation}
\label{DLR}
\mu(\theta|\Fcal_{\Lambda^{\rm c}})(\omega) = \pi_{\Lambda}^{\omega}(\theta)~\mbox{$\mu$-a.s.},
\end{equation}
for $\theta \in \Symb^{\Lambda}$, where $\{\pi^{\xi}_\Lambda\}_{\Lambda,\xi}$ is the stationary $\Z^d$ Gibbs specification for $\Phi$. 
\end{definition}

While our interactions and specifications are assumed to be shift-invariant, a Gibbs measure for such an interaction may or may not be stationary. The definition of n.n. Gibbs measure, shows that such a measure is an MRF. The definition is stated only for cylinder events $[\theta]$ in $\Lambda$, but this is equivalent to the usual definition with general events $A \in \mathcal{F}$ instead.

Every n.n. interaction $\Phi$ has at least one (stationary) n.n. Gibbs measure (special case of a general result in \cite[Theorem 3.7 and Theorem 4.2]{1-ruelle}). For a single $\Phi$, multiple Gibbs measures can exist. This phenomenon is usually called a \emph{phase transition}.

\subsection{Pressure}

Now we proceed to define the pressure of a n.n. interaction $\Phi$.

\begin{definition}
\label{pressure}
Given a n.n. interaction $\Phi$ for a set of restrictions $\mathcal{E}$, the \emph{pressure of $\Phi$} is defined as:
\begin{equation}
\Press(\Phi) := \lim_{n \rightarrow \infty} \frac{1}{|\block_n|} \log \PFunc.
\end{equation}

Given $n \in \N$, we can also define an analogous version $\hat{\mathrm{Z}}^\Phi_{\block_n}$ of the partition function $\PFunc$, but over globally admissible configurations:
\begin{equation}
\hat{\mathrm{Z}}^\Phi_{\block_n} := \sum_{\theta \in \Symb^{\block_n}: [\theta]^{\Omega(\mathcal{E})} \neq \emptyset} \exp\left(-\EFunc_{\block_n}(\theta)\right).
\end{equation}
\end{definition}

Notice that $\hat{\mathrm{Z}}^\Phi_{\block_n} \leq \mathrm{Z}^\Phi_{\block_n}$. The following result states that in the normalized limit, both quantities coincide.

\begin{theorem}[{\cite[Theorem 3.4]{1-ruelle}, see also \cite[Theorem 2.5]{1-friedland}}]
\label{friedland}
Given a n.n. interaction $\Phi$ for a set of restrictions $\mathcal{E}$:
\begin{equation}
\Press(\Phi) = \lim_{n \rightarrow \infty} \frac{1}{|\block_n|} \log \hat{\mathrm{Z}}^\Phi_{\block_n}.
\end{equation}
\end{theorem}

The pressure is the main quantity of interest in this paper. Our goals are to find simple representations of pressure in terms of very special configurations and use this to develop efficient (in principle) algorithms to approximate the pressure.


\section{Main models: Potts, Widom-Rowlinson and hard-core}
\label{section4}

In this section we introduce the three main families of lattice models studied in this paper. The first one will be the Potts model, which can be regarded as a generalization of the Ising model by considering more than two types of particles. The second one, the Widom-Rowlinson model, is also a multi-type particle system but with hard-core exclusion between particles of different type. The third one is the classical hard-core model.

\subsection{The (ferromagnetic) Potts model}

Given $d,q \in \N$ and $\beta > 0$, the \emph{$\Z^d$ (ferromagnetic) Potts model with $q$ types and inverse temperature $\beta$} is defined over the alphabet $\Symb_q = \{1,\dots, q\}$ and given by the n.n. interaction:
\begin{equation}
\Phi_\beta(\theta) =
\begin{cases}
-\beta	&	\mathrm{if}~\theta(x) = \theta(x+e_i),		\\
0		&	\mathrm{if}~\theta(x) \neq \theta(x+e_i),
\end{cases}
\end{equation}
for $\theta \in \Symb_q^{\{x,x+e_i\}}$, $x \in \Z^d$, $i = 1,\dots,d$, where the constraints $\mathcal{E}_i$ are empty. The specification $\pi^{\mathrm{FP}}_\beta = \{\pi_{\beta,\Lambda}^\xi\}_{\Lambda,\xi}$ induced by $\Phi_\beta$ defines the (ferromagnetic) Potts model, where neighbouring sites preferably align to each other with the same type or ``colour'' from the alphabet $\Symb_q$.

A measure $\mu \in \Mcal_{1}(\Symb_q^{\Z^d})$ is called a \emph{Potts Gibbs measure} for $q$ types and inverse temperature $\beta > 0 $ if it is a n.n. Gibbs measure for the specification $\pi^{\mathrm{FP}}_\beta$ above.

\begin{theorem}[{\cite{1-beffara}}]
\label{theoremPotts}
For the $\Z^2$ (ferromagnetic) Potts model with $q$ types and inverse temperature $\beta$, there exists a critical inverse temperature $ \beta_{\rm c}(q) := \log(1+\sqrt{q})$ such that uniqueness of Gibbs measures holds for $\beta < \beta_{\rm c}(q)$ and for $\beta > \beta_{\rm c}(q)$ there is a phase transition.
\end{theorem}

\subsection{The (multi-type) Widom-Rowlinson model}

Given $d,q \in \N$ and $\lambda > 0$, the \emph{$\Z^d$ Widom-Rowlinson model with $q$ types and activity $\lambda$} is defined over the alphabet $\SymbWR = \{0,1,\dots, q\}$, and given by the set of constraints $\mathcal{E} = (\mathcal{E}_1,\dots,\mathcal{E}_d)$, where $\mathcal{E}_i = \{\theta \in \left(\SymbWR \setminus \{0\}\right)^2: \theta(1) \neq \theta(2)\}$, for all $i = 1,\dots,d$, and by the n.n. interaction for $\mathcal{E}$ over configurations on sites:
\begin{equation}
\Phi_\lambda(\theta) =
\begin{cases}
-\log(\lambda)	&	\mathrm{if}~\theta \in \{1,\dots,q\},	\\
0			&	\mathrm{if}~\theta = 0,
\end{cases}
\end{equation}
where $\theta \in \SymbWR^{\{x\}}$ and $x \in \Z^d$. The specification $\pi^{\mathrm{WR}}_\lambda = \{\pi_{\lambda,\Lambda}^\xi\}_{\Lambda,\xi}$ induced by $\Phi_\lambda$ defines the (multi-type) Widom-Rowlinson model, where neighbouring sites are forced to align to each other with the same type or ``colour'' from the alphabet $\SymbWR$ or with $0$.

A measure $\mu \in \Mcal_{1}(\SymbWR^{\Z^d})$ is called a \emph{Widom-Rowlinson Gibbs measure} for $q$ types and activity $\lambda > 0$ if it is a n.n. Gibbs measure for the specification $\pi^{\mathrm{WR}}_\lambda$ above.

\begin{theorem}[{\cite{1-runnels}, see also \cite{3-georgii}}]
\label{theoremWR}
For the $\Z^2$ Widom-Rowlinson model with $q$ types and activity $\lambda$, uniqueness of Gibbs measures holds for sufficiently small $\lambda$ and there is a phase transition for sufficiently large $\lambda$.
\end{theorem}

\subsection{The hard-core lattice gas model}
\label{hardcore_sec}

Given $\gamma > 0$, the \emph{$\Z^d$ hard-core model with activity $\gamma$} is defined over the alphabet $\{0,1\}$, and given by the set of constraints $\mathcal{E}$, where $\mathcal{E}_i = \{(1,1)\}$, for all $i = 1,\dots,d$, and the the n.n. interaction for $\mathcal{E}$ over configurations on sites:
\begin{equation}
\Phi_\gamma(\theta) =
\begin{cases}
-\log(\gamma)	&	\mathrm{if}~\theta = 1,	\\
0			&	\mathrm{if}~\theta = 0,
\end{cases}
\end{equation}
for $\theta \in \{0,1\}^{\{x\}}$, $x \in \Z^d$. The specification $\pi^{\mathrm{HC}}_\gamma = \{\pi_{\gamma,\Lambda}^\xi\}_{\Lambda,\xi}$ induced by $\Phi_\gamma$ defines the hard-core model, where neighbouring sites cannot be both $1$.

A measure $\mu \in \Mcal_{1}(\{0,1\}^{\Z^d})$ is called a \emph{hard-core Gibbs measure} for activity $\gamma > 0$ if it is a n.n. Gibbs measure for the specification $\pi^{\mathrm{HC}}_\gamma$ above.

\begin{theorem}[{\cite[Theorem 3.3]{2-georgii}}]
\label{theoremHC}
For the $\Z^2$ hard-core model with activity $\gamma$, uniqueness of Gibbs measures holds for sufficiently small $\gamma$ and there is a phase transition for sufficiently large $\gamma$.
\end{theorem}

For both the Potts and Widom-Rowlinson models we will also distinguish a particular type of particle or colour in the alphabet. W.l.o.g., we can take the type $q$ in $\Symb_q$ or $\SymbWR \setminus \{0\}$, respectively. Given this colour, we will denote by $\omega_q$ the fixed point $q^{\Z^d}$. For the hard-core model, we will consider the two special points $\omega^{(e)}$ and $\omega^{(o)}$, given by:
\begin{equation}
\omega^{(e)}(x) :=
\begin{cases}
0	&	\mbox{if } \sum_i x_i \mbox{ is even,}		\\	
1	&	\mbox{if } \sum_i x_i \mbox{ is odd,}
\end{cases}
\end{equation}
and $\omega^{(o)} = \sigma_{e_1}(\omega^{(e)})$.


\section{Random-cluster models}
\label{section5}

The Potts and Widom-Rowlinson models have interpretations in terms of a random-cluster representation. The Potts model is related to a random-cluster model on bonds (via the so-called Edwards-Sokal coupling), while the Widom-Rowlinson is naturally related to a random-cluster model on sites.

\begin{definition}
A \emph{coupling} of two probability measures $\rho_1$ on a finite set $X$ and $\rho_2$ on a finite set $Y$, is a probability measure $\Prob$ on the set $X \times Y$ such that, for any $A \subseteq X$ and $B \subseteq Y$, we have that:
\begin{equation}
\Prob(A \times Y) = \rho_1(A) \mbox{ and } \Prob(X \times B) = \rho_2(B).
\end{equation}
\end{definition}

\subsection{The bond random-cluster model and the Potts model}
\label{bond_cluster}

We will make use of the \textit{bond random-cluster model}. One of our main results, Part I of Theorem \ref{potts-decay}, is proven using arguments based on this model. This model is a two parameter family of dependent bond percolation models on a finite graph. We are mainly interested in finite subgraphs of $\Z^2$ and we describe the model with boundary conditions indexed by $i = 0,1$.

Fix a finite simply lattice-connected set of sites $\Lambda$. Let $E^0(\Lambda)$ denote the set of bonds with both endpoints in $\Lambda$ and $E^1(\Lambda)$ the set of bonds with at least one endpoint in $\Lambda$. We speak of a bond $e$ as being \emph{open} if $w(e) = 1$, and as being \emph{closed} if $w(e) = 0$.

\begin{definition}
Given a finite simply lattice-connected set $\Lambda$, and parameters $p \in [0,1]$ and $q > 0$, we define the \emph{free} ($i=0$) and \emph{wired} ($i=1$) bond random-cluster distributions on ${E^i(\Lambda)}$ ($i = 0,1$) as the measures $\phi_{p,q,\Lambda}^{i}$ that to each $w \in \{0,1\}^{E^{i}(\Lambda)}$ assigns probability proportional to:
\begin{equation}
\phi_{p,q,\Lambda}^{(i)}(w) \propto \left\{\prod_{e \in E^i(\Lambda)} p^{w(e)}(1-p)^{1-w(e)}\right\}q^{k^i_\Lambda(w)} = \left(\frac{p}{1-p}\right)^{\#_1(w)} q^{k^i_\Lambda(w)},
\end{equation}
where $\#_1(w)$ is the number of open bonds in $w$ and $k_\Lambda^0(w)$ and $k_\Lambda^1(w)$ are the number of connected components (including isolated sites) in the graphs $(\Lambda, \{e \in E^0(\Lambda): w(e) = 1\})$ and $(\Z^2, E^0(\Z^2 \setminus \Lambda) \cup \{e \in E^1(\Lambda): w(e) = 1\})$, respectively.
\end{definition}

Notice that when $q = 1$, we recover the ordinary Bernoulli bond percolation measure $\phi_{p,\Lambda}$, while other choices of $q$ lead to dependence between bonds. For given $p$ and $q$, one can also define bond random-cluster measures $\phi_{p,q}^{(i)}$ on $\Z^2$ as a limit of finite volume measures $\phi_{p,q,\Lambda}^{(i)}$ ($i=0,1$).

\begin{theorem}[{\cite[Lemma 6.8]{2-georgii}}]
For $p \in [0,1]$ and $q \in \N$, the limiting measures:
\begin{equation}
\phi_{p,q}^{(i)} = \lim_{n \rightarrow \infty} \phi_{p,q,\Lambda_n}^{(i)},\qquad i \in \{0,1\},
\end{equation}
exist and are translation invariant, where $\{\Lambda_n\}_n$ is any increasing sequence of finite simply lattice-connected sets that exhausts $\Z^2$.
\end{theorem}

General bond random-cluster measures on $\Z^2$ can be defined using an analogue of the DLR condition \cite[Definition 4.29]{1-grimmett}. For $q \geq 1$, there is a value $p_{\rm c}(q)$ that delimits exactly the transition for existence of an infinite open cluster for these measures. It is known \cite[p. 107]{1-grimmett} that for $q \geq 1$ and $p < p_{\rm c}(q)$, there is a unique such measure which we denote by \label{phi-p-q}$\phi_{p,q}$ (characterized by the nonexistence of infinite open clusters), and that coincides with $\phi_{p,q}^{(0)}$ and $\phi_{p,q}^{(1)}$ in this region. It was recently proven (see \cite{1-beffara}) that $p_{\rm c}(q) = \frac{\sqrt{q}}{1 + \sqrt{q}}$, for every $q \geq 1$.

Let $p = 1-e^{-\beta}$. The {\em free Edwards-Sokal coupling} $\Prob_{p,q,\Lambda}^{(0)}$ (see \cite{1-grimmett}) is a coupling between the boundary-free Potts measure $\pi^{(f)}_{\beta,\Lambda}$ and $\phi_{p,q,\Lambda}^{(0)}$. The {\em wired Edwards-Sokal coupling} $\Prob_{p,q,\Lambda}^{(1)}$ is a coupling between $\pi_{\beta,\Lambda}^{\omega_q}$ and $\phi_{p,q,\Lambda}^{(1)}$. Notice that $p_{\rm c}(q) = 1-e^{-\beta_{\rm c}(q)}$.

These couplings are measures on pairs of site configurations and corresponding bond configurations. The projection to site configurations is the boundary-free/$\omega_q$-boundary Potts measure, and the projection to bond configurations is the free/wired bond random-cluster measure, respectively.

\begin{theorem}[{\cite[Theorem 1.13]{1-grimmett}}]
\label{edward}
Let $\Lambda$ be a finite simply lattice-connected set, $q \in \N$, and let $p \in [0,1]$ and $\beta > 0$ be such that $p = 1 - e^{-\beta}$. Then:
\begin{enumerate}
\item For $w \in \{0,1\}^{E^1(\Lambda)}$, the conditional measure $\Prob_{p,q,\Lambda}^{(1)}\left(\cdot \middle\vert \Symb_q^\Lambda \times \{w\}\right)$ on $\Symb_q^\Lambda$ is obtained by putting random colours on entire clusters of $w$ not connected with $\Z^2 \setminus \Lambda$ (of which there are $k^1_\Lambda(w) - 1$) and colour $q$ on the clusters connected with $\Z^2 \setminus \Lambda$. These colours are constant on given clusters, are independent between clusters, and the random ones are uniformly distributed on the set $\Symb_q$.
\item For $\theta \in \Symb^\Lambda$, the conditional measure $\Prob_{p,q,\Lambda}^{(1)}\left(\cdot \middle\vert \{\theta\} \times \{0,1\}^{E^1(\Lambda)}\right)$ on $\{0,1\}^{E^1(\Lambda)}$ is obtained as follows. Consider the extended configuration $\hat{\theta} = \theta q^{\partial \Lambda}$ and an arbitrary bond $e = \{x,y\} \in E^1(\Lambda)$. If $\hat{\theta}(x) \neq \hat{\theta}(y)$, we set $w(e) = 0$. If $\hat{\theta}(x) = \hat{\theta}(y)$, we set:
\begin{equation}
w(e) = \begin{cases}
1	&	\mathrm{with~probability~} p,	\\
0	&	\mathrm{otherwise},
\end{cases}
\end{equation}
the values of different $w(e)$ being (conditionally) independent random variables.
\end{enumerate}
\end{theorem}

The couplings can be used to relate probabilities and expectations for the Potts model to corresponding events and expectations in the associated bond random-cluster model. A main example is a relation between the two-point correlation function in the Potts model and the connectivity function in the bond random-cluster model \cite[Theorem 1.16]{1-grimmett}.

By considering a displaced version of $\Z^2$, namely $\frac{1}{2}\vecone + \Z^2$ (the \emph{dual lattice}), we can define a notion of duality for bond configurations $w$. Notice that every bond $e \in E(\Z^2)$ (if we think of bonds as unitary vertical and horizontal straight segments) is intersected perpendicularly by one and only one \emph{dual bond} $e^* \in E(\frac{1}{2}\vecone + \Z^2)$, so there is a clear correspondence between $E(\Z^2)$ and $E(\frac{1}{2}\vecone + \Z^2)$. We are mainly interested in wired bond random-cluster distributions on the set of sites $\tilde{\block}_n := [-n+1,n]^2 \cap \Z^2$. Given $n \in \N$, if we consider the set of bonds $E^1(\tilde{\block}_n)$, it is easy to check that there is a correspondence $e \mapsto e^*$ between this set and the set of bonds from $\frac{1}{2}\vecone + \Z^2$ with both endpoints in $[-n,n]^2 \cap \left(\frac{1}{2}\vecone + \Z^2\right)$, which can be identified with the set $E^{0}(\block_n)$. Then, given a bond configuration $w \in E^1(\tilde{\block}_n)$ we can associate a dual bond configuration $w^* \in E^{0}(\block_n)$ such that $w^*(e^*) = 0$ if and only if $w(e) = 1$.

Considering this, we have the corresponding equality:
\begin{proposition}[{\cite[Equation 6.12 and Theorem 6.13]{1-grimmett}}]
\label{RC_dual_prob}
Given $n \in \N$, $p \in [0,1]$ and $q \in \N$:
\begin{equation}
\label{free_wired_probs}
\phi^{(1)}_{p,q,\tilde{\block}_n}(w) = \phi^{(0)}_{p^*,q,\block_n}(w^*),
\end{equation}
for any bond configuration $w \in \{0,1\}^{E^{1}(\tilde{\block}_n)}$, where $\tilde{\block}_n = [-n+1,n]^2 \cap \Z^2$ and $p^* \in [0,1]$ is the dual value of $p$, which is given by:
\begin{equation}
\label{dual_prob}
\frac{p^*}{1 - p^*} = \frac{q(1 - p)}{p}.
\end{equation}
\end{proposition}

The previous duality result can be generalized to more arbitrary shapes and it has also a counterpart from free-to-wired boundary conditions, instead of from wired-to-free.

The unique fixed point of the map $p \mapsto p^*$ defined by (\ref{dual_prob}) is $\frac{\sqrt{q}}{1+\sqrt{q}}$ and, as mentioned above, is known to coincide with the critical point $p_{\rm c}(q)$ for the existence of an infinite open cluster for the bond random-cluster model (see \cite[Theorem 1]{1-beffara}). It is easy to see that $p > p_{\rm c}(q)$ iff $p^* < p_{\rm c}(q)$.

\subsection{The site random-cluster model and the Widom-Rowlinson model}

In a similar fashion to the bond random-cluster model, we can perturb Bernoulli site percolation, where the probability measure is changed in favour of configurations with many (for $q > 1$) or few (for $q < 1$) connected components. The resulting model is called the \emph{site random-cluster model}.

\begin{definition}
Given $\Lambda \Subset \Z^2$, and parameters $p \in [0,1]$ and $q > 0$, the \emph{wired site random-cluster measure} $\psi^{(1)}_{p,q,\Lambda}$ is the probability measure on $\{0,1\}^{\Lambda}$ which to each $\theta \in \{0,1\}^{\Lambda}$ assigns probability proportional to:
\begin{equation}
\label{siteRC}
\psi^{(1)}_{p,q,\Lambda}(\theta) \propto \left\{\prod_{x \in \Lambda} p^{\theta(x)}(1-p)^{1-\theta(x)}\right\} q^{\kappa_\Lambda(\theta)} =\lambda^{\#_1(\theta)} q^{\kappa_\Lambda(\theta)},
\end{equation}
where $\lambda = \frac{p}{1-p}$, $\#_1(\theta)$ is the number of $1$'s in $\theta$ and $\kappa_\Lambda(\theta)$ is the number of connected components in $\{x \in \Lambda: \theta(x) = 1\}$ that do not intersect $\underline{\partial} \Lambda$.
\end{definition}

The \emph{free site random-cluster measure} $\psi^{(0)}_{p,q,\Lambda}$ is defined as in (\ref{siteRC}) by replacing $\kappa_\Lambda(\theta)$ by the total number of connected components in $\Lambda$. However, we will not require that measure in this work. In any case, taking $q = 1$ gives the ordinary Bernoulli site percolation $\psi_{p,\Lambda}$, while other choices of $q$ lead to dependence between sites, similarly to the bond random-cluster model.

\begin{proposition}
\label{WR_cluster}
Given a set $\Lambda \Subset \Z^2$ and parameters $\lambda > 0$ and $q \in \N$, consider the Widom-Rowlinson with $q$ types distribution and monochromatic boundary condition $\pi_{\lambda,\Lambda}^{\omega_q}$. Now, let $f:\SymbWR^\Lambda \to \{0,1\}^\Lambda$ be defined site-wise as:
\begin{equation}
(f(\theta))(x) =
\begin{cases}
0	&	\mbox{if } \theta(x) = 0,		\\
1	&	\mbox{if } \theta(x) \neq 0,
\end{cases}
\end{equation}
for $\theta \in \SymbWR^\Lambda$ and $x \in \Lambda$, and let $p = \frac{\lambda}{1+\lambda}$. Then, $f_*\pi_{\lambda,\Lambda}^{\omega_q} = \psi^{(1)}_{p,q,\Lambda}$, where $f_*\pi_{\lambda,\Lambda}^{\omega_q}(\cdot) := \pi_{\lambda,\Lambda}^{\omega_q}(f^{-1}(\cdot))$ denotes the push-forward measure on $\{0,1\}^\Lambda$.
\end{proposition}

The requirement that $\kappa_\Lambda(\cdot)$ does not count connected components that intersect the inner boundary of $\Lambda$ in the site random-cluster model, corresponds to the fact that non $0$ sites adjacent to the monochromatic boundary $\omega_q(\partial \Lambda)$ in the Widom-Rowlinson model must have the same colour $q$.

For $q=2$, Proposition \ref{WR_cluster} is proven in \cite[Lemma 5.1 (ii)]{1-higuchi}, and the proof extends easily for general $q$. Proposition \ref{WR_cluster} can be regarded as a coupling between $\pi_{\lambda,\Lambda}^{\omega_q}$ and $\psi^{(1)}_{p,q,\Lambda}$, because a push-forward measure can be naturally coupled with the original measure.

It is important to notice that $\psi^{(1)}_{p,q,\Lambda}$ is itself not an MRF: given sites on a simple circuit $\Circ$, the inside and outside of $\Circ$ are generally not conditionally independent, because knowledge of sites outside $\Circ$ could cause connected components of $1$'s in $\Circ$ to ``amalgamate'' into a single component, which would affect the conditional distribution of configurations inside $\Circ$. The following lemma shows that in certain situations, when conditioning on a circuit $\Circ$ labeled entirely by $1$'s, this kind of amalgamation does not occur.

\begin{lemma}
\label{pseudoMRF}
Let $\emptyset \neq \Theta \subseteq \Lambda \Subset \Z^2$ be such that $\Lambda^c \cup \overline{\Theta}^\star$ is connected. Take $\Delta := \partial^\star \Theta \cap \Lambda$. Consider an event $A \in \mathcal{F}_\Theta$ and a configuration $\tau \in \{0,1\}^{\Sigma}$, where $\Sigma \subseteq \Lambda \setminus \overline{\Theta}^\star$. Then:
\begin{equation}
\psi^{(1)}_{p,q,\Lambda}(A \vert 1^{\Delta} \tau) = \psi^{(1)}_{p,q,\Lambda}(A \vert 1^{\Delta} 0^{\Lambda \setminus \overline{\Theta}^\star}).
\end{equation}
\end{lemma}

\begin{proof}
W.l.o.g., we may assume that  $A$ is a cylinder event $[\theta]$ with $\theta \in \{0,1\}^\Theta$ (by linearity) and $\Sigma = \Lambda \setminus \overline{\Theta}^\star$ (by taking weighted averages).

Now, $\Sigma = \Lambda \setminus \overline{\Theta}^\star$ can be written as a disjoint union of $\star$-connected components $\Sigma = K_1 \sqcup \cdots \sqcup K_n$. For every $i$, $\partial^\star K_i \subseteq \Lambda^c \cup \overline{\Theta}^\star$ (in fact, $\partial^\star K_i \subseteq \Lambda^c \cup \Delta$). Since $\Lambda^c \cup \overline{\Theta}^\star$ is connected and $\Lambda$ is finite, for every site in $\partial^\star K_i$ there is a path to infinity that does not intersect $K_i$.

Then, by application of a result of Kesten (see \cite[Lemma 2.23]{1-kesten}), $\partial^\star K_i$ is connected, for every $i$. In addition, we have that $\Lambda = \Theta \sqcup \Delta \sqcup \Sigma$ and $\partial^\star K_i \subseteq \Lambda^c \cup \Delta$.

We claim that:
\begin{equation}
\label{decomposition}
\kappa_{\Lambda}(\upsilon) =  \kappa_{\Lambda}(\upsilon(\Theta)1^{\Delta} 0^{\Sigma}) + \sum_{i=1}^n\kappa_{K_i}(\upsilon(K_i)) = \kappa_{\Lambda}(\upsilon(\Theta)1^{\Delta} 0^{\Sigma}) + \kappa_{\Sigma}(\tau),
\end{equation}
for any $\upsilon \in \{0,1\}^\Lambda$ such that $\upsilon(\Delta) = 1^{\Delta}$ and $\upsilon(\Sigma) = \tau$.

\begin{figure}[ht]
\label{lemma1-pic}
\centering
\includegraphics[scale = 0.65]{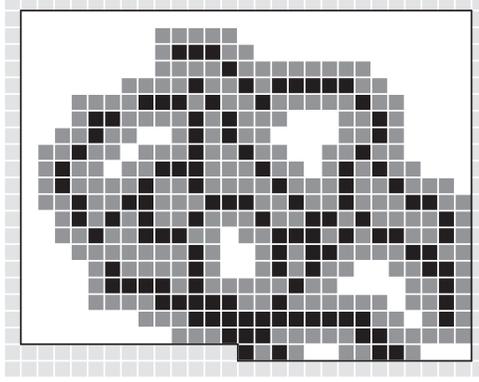}
\caption{A $\star$-connected $\Theta$ (in black), the set $\Delta = \partial^\star \Theta \cap \Lambda$ (in dark grey) and $\Lambda^c$ (in light grey) for $\Lambda = S_{y,z}$.}
\end{figure} 

To see this, given such $\upsilon$, we exhibit a bijection $r$ between the connected components of $\upsilon$ that do not intersect $\underline{\partial} \Lambda$ and the union of: (a) the connected components of $\upsilon(\Theta)1^\Delta 0^\Sigma$ that do not intersect $\underline{\partial} \Lambda$, and (b) the connected components of $\upsilon(K_i)$ that do not intersect $\underline{\partial} K_i$, for all $i$; namely, if $C \subseteq \Lambda$ is a connected component of $\upsilon$, then $r$ is defined as follows:
\begin{equation}
r(C) =
\begin{cases}
C \cap \overline{\Theta}^\star	& \mbox{ if } C \cap \overline{\Theta}^\star \neq \emptyset,		\\
C						& \mbox{ if } C \subseteq \Sigma.
\end{cases}
\end{equation}

In order to see that $r$ is well-defined, note that if $C$ intersects $\overline{\Theta}^\star$ and $\Sigma$, the set $C \cap \overline{\Theta}^\star$ is still connected thanks to the fact that $\partial^\star K_i$ is connected and $\upsilon(\Delta) = 1^{\Delta}$. To see that $r$ is onto, observe that if $C'$ is a connected component of $\upsilon(\Theta)1^\Delta 0^\Sigma$, then there is a unique component $C$ of $\upsilon$ such that $C \cap \overline{\Theta}^\star = C'$, due again to the fact that $\partial^\star K_i$ is connected. And $r$ is clearly injective because two distinct connected components cannot intersect.

Finally, we conclude from (\ref{decomposition}) that:
\begin{align}
\psi^{(1)}_{p,q,\Lambda}(\theta \ | \ 1^{\Delta} \tau)	&	=	\frac{\lambda^{\#_1(\theta 1^{\Delta} \tau)} q^{\kappa_{\Lambda}(\theta 1^{\Delta} \tau)}}{\sum_{\upsilon \in \{0,1\}^{\Lambda}: \upsilon(\Delta) = 1^{\Delta}, \upsilon(\Sigma) = \tau} \lambda^{\#_1(\upsilon)} q^{\kappa_{\Lambda}(\upsilon)}}	\\
										&	=	\frac{\lambda^{\#_1(\theta 1^\Delta) + \#_1(\tau)} q^{\kappa_{\Lambda}(\theta 1^\Delta 0^{\Sigma}) + \kappa_{\Sigma}(\tau)}}{\sum_{\upsilon \in \{0,1\}^{\Lambda}: \upsilon(\Sigma) = \tau} \lambda^{\#_1(\upsilon(\Theta) 1^\Delta) + \#_1(\tau)} q^{\kappa_{\Lambda}(\upsilon(\Theta) 1^\Delta 0^{\Sigma}) + \kappa_{\Sigma}(\tau)}}	\\
										&	=	\frac{\lambda^{\#_1(\theta 1^\Delta)} q^{\kappa_{\Lambda}(\theta 1^\Delta 0^{\Lambda \setminus \overline{\Theta}^\star})}}{\sum_{\tilde{\theta} \in \{0,1\}^{\Theta}} \lambda^{\#_1(\tilde{\theta} 1^\Delta)} q^{\kappa_{\Lambda}(\tilde{\theta} 1^\Delta 0^{\Lambda \setminus \overline{\Theta}^\star})}}	=	\psi^{(1)}_{p,q,\Lambda}(\theta \vert 1^{\Delta} 0^{\Lambda \setminus \overline{\Theta}^\star}),
\end{align}
as we wanted.
\end{proof}

\begin{remark}
\label{notePseudoMRF}
We claim that if $\emptyset \neq \Theta \subseteq \Lambda \Subset \Z^2$ are such that $\Lambda^c$ is connected, $\Theta$ is $\star$-connected and $\overline{\Theta}^\star \cap \underline{\partial} \Lambda \neq \emptyset$, then $\Lambda^c \cup \overline{\Theta}^\star$ is connected, which is the main hypothesis of Lemma \ref{pseudoMRF}. This follows from the easy fact that the $\star$-closure of
a $\star$-connected set is connected.
\end{remark}


\section{Pressure representation}
\label{section6}

\subsection{Variational principle}

The variational principle states that the pressure of an interaction has a variational characterization in terms of shift-invariant measures. We state the variational principle below for the case of an n.n. interaction $\Phi$ for a set of restrictions $\mathcal{E}$.

\begin{theorem}[Variational principle \cite{1-keller,1-misiurewicz,1-ruelle}]
\label{VarPrin}
Given a n.n. interaction $\Phi$ for a set of restrictions $\mathcal{E}$, we have that:
\begin{equation}
\Press(\Phi) = \sup_{\mu \in \mathcal{M}_{1,\sigma}(\Omega(\mathcal{E}))}\left(h(\mu) + \int{A_\Phi}d\mu\right),
\end{equation}
where:
\begin{itemize}
\item $A_\Phi(\omega) :=  -\Phi\left(\omega(\veczero)\right) - \sum_{i=1}^{d} \Phi\left(\omega(\{\veczero,e_i\})\right)$, for $\omega \in \Omega(\mathcal{E})$, and
\item $h(\mu) := \lim_{n \rightarrow \infty} \frac{-1}{|\block_n|} \sum_{\theta \in \Symb^{\block_n}} \mu(\theta)\log(\mu(\theta))$ is the \emph{measure-theoretic entropy} of $\mu$, where $0 \log 0 = 0$.
 \end{itemize}
\end{theorem}

In this case, the supremum is also always achieved (see \cite[Section 4.2]{1-keller}) and any measure which achieves the supremum is called an \emph{equilibrium state} for $A_\Phi$. So, if $\mu$ is an equilibrium state, then:
\begin{equation}
\Press(\Phi) = h(\mu) + \int{A_\Phi}d\mu.
\end{equation}

For a shift-invariant measure $\mu$ and $\Lambda \Subset \Z^d \setminus \{\veczero\}$, define:
\begin{equation}
p_{\mu, \Lambda}(\omega):= \mu(\omega(\veczero) | \omega(\Lambda)),
\end{equation}
and let $p_\mu(\omega) := \lim_{n \rightarrow \infty} p_{\mu, \block_n \cap \past}(\omega)$, which exists $\mu$-a.s.~\cite[Theorem 3.1.10]{1-keller} by L\'evy's zero-one law. In addition, let:
\begin{equation}
I_\mu(\omega): = - \log p_\mu(\omega),
\end{equation}
which is also defined $\mu$-a.s. and is usually called the \emph{information function}. It is well-known (see \cite[p. 318, Equation 15.18]{1-georgii} or \cite[Theorem 2.4, p. 283]{1-krengel}) that for any shift-invariant measure $\mu$, $h(\mu) = \int{I_\mu}d\mu$. Therefore, if $\mu$ is an equilibrium state for $\Phi$, we can rewrite the preceding formula for $\Press(\Phi)$ as:
\begin{equation}
\Press(\Phi) = \int{\left(I_\mu + A_\Phi\right)}d\mu.
\end{equation}

So, the pressure can be represented as the integral of a function, determined by an equilibrium state $\mu$ and $\Phi$, with respect to $\mu$.

In this section, we show that the pressure can be represented as the integral of a function similar to $I_\mu + A_\Phi$, with respect to any invariant measure $\nu$, assuming some conditions. This is useful for approximation of pressure when $\nu$ is an atomic measure supported on a periodic configuration (see Section \ref{section9}).

One of the conditions involves the SFT $\Omega(\mathcal{E})$.

\begin{definition}
\label{Dcond}
A n.n. SFT $\Omega(\mathcal{E})$ for a set of constraints $\mathcal{E}$ satisfies the \emph{square block D-condition} if there exists a sequence of integers $\{r_n\}_{n \geq 1}$ such that $\frac{r_n}{n} \to 0$ as $n \to \infty$ and, for any finite set $\Lambda \Subset \block_{n+r_n}^c$, $\theta \in \Symb^{\block_n}$ and $\tau \in \Symb^\Lambda$:
\begin{equation}
[\theta]^{\Omega(\mathcal{E})}, [\tau]^{\Omega(\mathcal{E})} \neq \emptyset \implies [\theta\tau]^{\Omega(\mathcal{E})} \neq \emptyset.
\end{equation}
\end{definition}

This condition is a strengthened version of the classical D-condition (see \cite[Section 4.1]{1-ruelle}) which guarantees that the set of Gibbs measures for $\Phi$ coincides with the set of equilibrium states for $A_\Phi$.

\begin{definition}
Given a set of restrictions $\mathcal{E}$, the corresponding n.n. SFT $\Omega(\mathcal{E}) \subseteq \Symb^{\Z^d}$ and $a \in \Symb$, we say that $\Omega(\mathcal{E})$ has a \emph{safe symbol} $a$ if $(a,b), (b,a) \notin \mathcal{E}_i$, for every $b \in \Symb$, for all $i = 1,\dots,d$.
\end{definition}

It is easy to see that if $\Omega(\mathcal{E})$ has a safe symbol, then it satisfies the square block D-condition. For the sets of restrictions $\mathcal{E}$ in the Potts, Widom-Rowlinson and hard-core models, the corresponding n.n. SFT $\Omega(\mathcal{E})$ has a safe symbol in each case (any $a \in \Symb_q$, $0 \in \SymbWR$, and $0 \in \{0,1\}$, respectively), so $\Omega(\mathcal{E})$ satisfies the square block D-condition for the three models.

\subsection{The function $\hat{\pi}$ and additional notation}
\label{I-hat}

Given a n.n. interaction $\Phi$ for a set of constraints $\mathcal{E}$, we will define some useful functions from $\Omega(\mathcal{E})$ to $\R$. First, given $\veczero \in \Lambda \Subset \Z^d$ and $\omega \in \Omega(\mathcal{E})$, we define:
\begin{equation}
\pi_{\Lambda}(\omega) := \pi_{\Lambda}^{\omega}(\theta(\veczero) = \omega(\veczero)) = \pi_{\Lambda}^{\omega(\partial \Lambda)}(\theta(\veczero) = \omega(\veczero)).
\end{equation}

Recall that, for $y,z \in \Z^d$ such that $y,z \geq \veczero$, we have defined the set $S_{y,z}$ as $\{x \succcurlyeq \veczero: -y \leq x \leq z\}$. Now, given $y,z \geq \veczero$ and $\omega \in \Omega(\mathcal{E})$, define $\pi_{y,z}(\omega) := \pi_{S_{y,z}}(\omega)$ and, given $n \in \N$, abbreviate $\pi_n(\omega) := \pi_{\vecone n, \vecone n}(\omega)$. Considering this, we also define the limit $\hat{\pi}(\omega) := \lim_{n \rightarrow \infty} \pi_n(\omega)$, whenever it exists. If such limit exists, we will also denote $\hat{I}_{\pi}(\omega) := -\log \hat{\pi}(\omega)$.

It is not difficult to prove that under some mixing assumptions over an MRF $\mu$, namely the SSM property introduced in Definition \ref{SSMspec} (see Section \ref{section7}), and assuming that $\supp(\mu) = \Omega(\mathcal{E})$, one has that the original information function $I_\mu$ coincides with $\hat{I}_\pi$ in $\Omega(\mathcal{E})$. In this sense, our definition provides a generalization of previous results (see \cite{1-gamarnik}), where $I_\mu$ may not be even well-defined.

Now, suppose we have a shift-invariant measure $\nu$ such that $\supp(\nu) \subseteq \Omega(\mathcal{E})$. We say that:
\begin{equation}
\lim_{y,z \rightarrow \infty} \pi_{y,z}(\omega) = \hat{\pi}(\omega) \mbox{ uniformly over } \omega \in\supp(\nu),
\end{equation}
if for all $\varepsilon > 0$, exists $k \in \N$ such that for all $\forall y,z \geq \vecone k$:
\begin{equation}
\left| \pi_{y,z}(\omega) - \hat{\pi}(\omega) \right| < \varepsilon,\mbox{ for all $\omega \in \supp(\nu)$}.
\end{equation}

In addition, we introduce the following bound:
\begin{equation}
\mathrm{c}_\pi(\nu) := \inf\{\pi_{\Lambda}(\omega): \veczero \in \Lambda \Subset \Z^d, \omega \in \supp(\nu)\}.
\end{equation}

\begin{lemma}
\label{safeCmu}
Let $\pi$ be a n.n. interaction $\Phi$ for a set of restrictions $\mathcal{E}$, with $\pi$ and $\Omega(\mathcal{E})$ the corresponding specification and n.n. SFT. Then, if $\Omega(\mathcal{E})$ has a a safe symbol, we have that $\mathrm{c}_\pi(\nu) > 0$, for any shift-invariant measure $\nu$ such that $\supp(\nu) \subseteq \Omega(\mathcal{E})$.
\end{lemma}

\begin{proof}
The proof is analogous and a particular case of \cite[Proposition 2.17]{1-marcus}. In that reference, under these assumptions, it is shown that $\mathrm{c}_\mu(\nu) := \inf\{\mu(\omega(\veczero) \vert \omega(\Lambda)): \Lambda \Subset \Z^d \setminus \{\veczero\}, \omega \in \supp(\nu)\} > 0$, for a given n.n. Gibbs measure $\mu$ for $\Phi$. We leave it to the reader to verify that $\mathrm{c}_\pi(\nu) \geq \mathrm{c}_\mu(\nu)$, for any such $\mu$.
\end{proof}

In fact, much weaker conditions than the existence of a safe symbol are sufficient for the result of Lemma \ref{safeCmu} and also for having the square block D-condition. See, for example, the \emph{single-site fillability} property \cite{1-marcus} and the \emph{topological strong spatial mixing property} \cite{1-briceno}. Notice that, since the Potts, Widom-Rowlinson and hard-core models have a safe symbol, we have that $c_\pi(\nu) > 0$, for any shift-invariant $\nu$ with $\supp(\nu) \subseteq \Omega(\mathcal{E})$.

\subsection{Pressure representation theorem}

Pressure representation results can be found in \cite[Theorems 3.1 and 3.6]{1-marcus}. Those results are not adequate for the application to the specific models we are considering in this paper. Instead we will use the following result, whose proof is adapted from the proof of \cite[Theorem 3.1]{1-marcus}, as well as an idea of \cite[Theorem 3.6]{1-marcus}. In contrast to the results of \cite{1-marcus}, our result makes assumptions on the specification rather than a Gibbs measure.

\begin{theorem}
\label{press-rep}
Let $\Phi$ be a n.n. interaction for a set of restrictions $\mathcal{E}$ and suppose that $\Omega(\mathcal{E})$ satisfies the square block D-condition. Let $\nu$ be a shift-invariant measure such that $\supp(\nu) \subseteq \Omega(\mathcal{E})$ and $\mathrm{c}_\pi(\nu) > 0$. In addition, suppose that:
\begin{equation}
\label{star}
\lim_{y,z \rightarrow \infty} \pi_{y,z}(\omega) = \hat{\pi}(\omega) \mbox{ uniformly over }~\omega \in\supp(\nu).
\end{equation}

Then:
\begin{equation}
\label{PresRepEqn}
\Press(\Phi) = \int{\left(\hat{I}_{\pi} + A_\Phi\right)}d\nu.
\end{equation}
\end{theorem}

\begin{proof}
Choose $\ell < 0$ and $L > 0$ to be lower and upper bounds respectively on values of $\Phi$. Given $n \in \N$, let $r_n$ be as in the definition of the square block D-condition and consider the sets $\block_n$ and $\Lambda_n := \block_{n + r_n}$. We begin by proving that:
\begin{equation}
\label{unifconverg}
\frac{1}{|\block_n|} (\log \PFunc + \log \pi^{\omega}_{\Lambda_n}(\omega(\block_n)) + \EFunc_{\block_n}(\omega(\block_n))) \rightarrow 0,
\end{equation}
uniformly in $\omega \in \Omega$. For this, we will only use the square block D-condition. We fix $n \in \N$, $\omega \in \supp(\nu)$ and let $m_n := |\Lambda_n| - |\block_n|$. Let $C_d \geq 1$ be a constant such that for any $\Delta \Subset \Z^d$, the total number of sites and bonds contained in $\overline{\Delta}$ is bounded from above by $C_d |\Delta|$. 

\begin{align}
\pi^{\omega}_{\Lambda_n}(\omega(\block_n))	&	\geq	\pi^{\omega}_{\Lambda_n}(\omega(\Lambda_n))	\\
									&	=	\frac{\exp({-\EFunc_{\overline{\Lambda_n}}(\omega(\overline{\Lambda_n}))})}{\sum_{\theta: \theta \omega(\partial \Lambda_n) \mathrm{~feasible}} \exp({-E_{\overline{\Lambda_n}}^\Phi(\theta \omega(\partial \Lambda_n))})}	\\
									&	\geq	\frac{\exp({-\EFunc_{\block_n}(\omega(\block_n)) - C_dm_nL})}{\sum_{\tau \in \Symb^{\block_n}: \tau \mathrm{~feasible}} \exp({-\EFunc_{\block_n}(\tau)}) |\Symb|^{C_dm_n}\exp({-C_dm_n\ell})}	\\
									&	=	\frac{\exp({-\EFunc_{\block_n}(\omega(\block_n))})}{\PFunc}\exp({m_n(C_d\ell - C_dL-C_d\log|\Symb|)}).
\end{align}

Now, if $\tau_{\max}$ achieves the maximum of $\pi^{\omega}_{\Lambda_n}(\omega(\block_n)\tau)$ over $\tau \in \Symb^{\Lambda_n \setminus \block_n}$, then:
\begin{align}
\pi^{\omega}_{\Lambda_n}(\omega(\block_n))	&	=	\sum_{\tau \in \Symb^{\Lambda_n \setminus \block_n}: \omega(\block_n)\tau \mathrm{~feasible}} \pi^{\omega}_{\Lambda_n}(\omega(\block_n)\tau)		\\
									&	\leq	|\Symb|^{m_n} \pi^{\omega}_{\Lambda_n}(\omega(\block_n)\tau_{\max})	\\		
									&	=	|\Symb|^{m_n} \frac{\exp({-\EFunc_{\overline{\Lambda_n}}(\omega(\block_n)\tau_{\max}\omega(\partial\Lambda_n))})}{\sum_{\theta: \theta \omega(\partial \Lambda_n) \mathrm{~feasible}} \exp({-\EFunc_{\overline{\Lambda_n}}(\theta \omega(\partial \Lambda_n))})}	\\
									&	\leq	|\Symb|^{m_n}\frac{\exp({-\EFunc_{\block_n}(\omega(\block_n)) - C_dm_n\ell})}{\sum_{\tau \in \Symb^{\block_n}: [\tau]^\Omega \neq \emptyset} e^{-\EFunc_{\block_n}(\tau)} \exp({-C_dm_nL})} \label{DcondIneq}	\\
									&	\leq	\frac{\exp({-\EFunc_{\block_n}(\omega({\block_n}))})}{\hat{\mathrm{Z}}^\Phi_{\block_n}}\exp({-m_n(C_d\ell - C_dL-C_d\log|\Symb|)}),
\end{align}
where the square block D-condition has been used in (\ref{DcondIneq}). Therefore,
\begin{equation}
\alpha^{-m_n} \leq \pi^{\omega}_{\Lambda_n}(\omega(\block_n)) \PFunc \exp({\EFunc_{\block_n}(\omega(\block_n))}) \leq \frac{\mathrm{Z}^\Phi_{\block_n}}{\hat{\mathrm{Z}}^\Phi_{\block_n}} \alpha^{m_n},
\end{equation}
where $\alpha := e^{-(C_d\ell - C_dL-C_d\log|\Symb|)}$. Since $\frac{m_n}{|\block_n|} \rightarrow 0$ and $\frac{1}{|\block_n|}\left(\log\mathrm{Z}^\Phi_{\block_n} - \log\hat{\mathrm{Z}}^\Phi_{\block_n}\right) \rightarrow 0$ (thanks to Theorem \ref{friedland}), we have obtained (\ref{unifconverg}).

We use (\ref{unifconverg}) to represent pressure:
\begin{align}
\Press(\Phi)	&	= \lim_{n \rightarrow \infty} \frac{\log \PFunc}{|\block_n|} = \lim_{n \rightarrow \infty} \int{\frac{\log \PFunc}{|{\block_n}|}}d\nu	\\
			&	= \lim_{n \rightarrow \infty} \int{\frac{-\log \pi^{\omega}_{\Lambda_n}(\omega({\block_n})) - \EFunc_{\block_n}(\omega({\block_n}))}{|{\block_n}|}}d\nu.
\end{align}

(Here the second equality comes from the fact that $\frac{\log \PFunc}{|{\block_n}|}$ is independent of $\omega$, and the third from (\ref{unifconverg}).) Since $\nu$ is shift-invariant, it can be checked that:
\begin{equation}
\lim_{n \rightarrow \infty} \int{\frac{- \EFunc_{\block_n}(\omega({\block_n}))}{|{\block_n}|}}d\nu = \int{A_\Phi}d\nu,
\end{equation}
and so we can write:
\begin{equation}
\Press(\Phi) = \int{A_\Phi}d\nu - \lim_{n \rightarrow \infty} \int \frac{\log \pi^{\omega}_{\Lambda_n}(\omega({\block_n}))}{|{\block_n}|}d\nu.
\end{equation}

It remains to show that:
\begin{equation}
\lim_{n \rightarrow \infty} \int \frac{-\log \pi^{\omega}_{\Lambda_n}(\omega({\block_{n}}))}{|{\block_{n}}|}d\nu= \int{\hat{I}_\pi}d\nu.
\end{equation}

Fix $\omega \in \supp(\nu)$ and denote $c := \mathrm{c}_\pi(\nu)$. We will decompose $\pi^{\omega}_{\Lambda_n}(\omega({\block_{n}}))$ as a product of conditional probabilities. By (\ref{star}), for any $\varepsilon > 0$, there exists $k := k_{\varepsilon}$ so that for $y,z \geq \vecone k$, $|\pi_{y,z}(\omega) - \hat{\pi}(\omega)| < \varepsilon$ for all $\omega \in \supp(\nu)$. For $x \in \block_{n-1}$, we denote $\block_n^{-}(x) := \left\{y \in \block_{n-1}: y \prec x\right\}$. Then, we can decompose $\pi^{\omega}_{\Lambda_n}(\omega({\block_{n}}))$ as:
\begin{align}
\label{bigdecomp}
\pi^{\omega}_{\Lambda_n}(\omega({\block_{n}}))	&	=	\pi^{\omega}_{\Lambda_n}\left(\omega(\underline{\partial} \block_{n})\right)\prod_{x \in \block_{n-1}} \pi^{\omega}_{\Lambda_n}\left(\omega(x) \middle\vert \omega\left(\block_n^{-}(x) \cup \underline{\partial} \block_{n}\right)\right)\\
										&	=	\pi^{\omega}_{\Lambda_n}\left(\omega(\underline{\partial} \block_{n})\right)\prod_{x \in \block_{n-1}} \pi_{y(x), z(x)}(\sigma_x(\omega)),
\end{align}
where $y(x) := \vecone n + x$ and $z(x) := \vecone n - x$, thanks to the MRF property and stationarity of the specification.

Let's denote $R_{n,k} := \block_{n} \setminus \block_{n-k}$. Then, $\block_{n} = \underline{\partial} \block_{n} \sqcup \block_{n-k-1} \sqcup R_{n-1,k}$ and we have:
\begin{align}
c^{|\underline{\partial} \block_{n}| + |R_{n-1,k}|}\prod_{x \in \block_{n-k-1}} \pi_{y(x), z(x)}(\sigma_x(\omega))		&	\leq	\pi^{\omega}_{\Lambda_n}(\omega({\block_{n}}))	\\
																					&	\leq   \prod_{x \in \block_{n-k-1}} \pi_{y(x), z(x)}(\sigma_x(\omega)).
\end{align}

Taking $-\log(\cdot)$, we have that:
\begin{align}
0	&	\leq -\log \pi^{\omega}_{\Lambda_n}(\omega({\block_{n}})) - \sum_{x \in \block_{n-k-1}} -\log \pi_{y(x), z(x)}(\sigma_x(\omega))	\\
	&	\leq (|\underline{\partial} \block_{n}| + |R_{n-1,k}|)\log \left(c^{-1}\right).
\end{align}

So, by the choice of $k$, for $x \in \block_{n-k-1}$,
\begin{align}
\left|\pi_{y(x), z(x)}(\sigma_x(\omega)) - \hat{\pi}(\sigma_x(\omega))\right| < \varepsilon,
\end{align}
and since $\pi_{y(x), z(x)}(\sigma_x(\omega)), \hat{\pi}(\sigma_x(\omega)) \geq c > 0$, by the Mean Value Theorem:
\begin{align}
\left|-\log \pi_{y(x), z(x)}(\sigma_x(\omega)) - \hat{I}_\pi(\sigma_x(\omega))\right| < \varepsilon c^{-1},
\end{align}

It follows from (\ref{star}) that $\hat{\pi}$ is the uniform limit of continuous functions on $\supp(\nu)$. In addition, $\hat{\pi}(\omega) \geq c > 0$, for all $\omega \in \supp(\nu)$. Therefore, we can integrate with respect to $\nu$ to see that:
\begin{equation}
\label{closesites}
\left|\int{-\log \pi_{y(x), z(x)}(\sigma_x(\omega))}d\nu - \int \hat{I}_{\pi}(\omega) d\nu\right| < \varepsilon c^{-1}.
\end{equation}

We now combine the previous equations to see that:
\begin{align}
\left|\int -\log \pi^{\omega}_{\Lambda_n}(\omega({\block_{n+1}})) d\nu - \int \hat{I}_{\pi}(\omega) d\nu |\block_{n-k-1}| \right|	\\
\leq |\block_{n-k-1}| \varepsilon c^{-1} + (|\underline{\partial} \block_{n}| + |R_{n-1,k}|)\log \left(c^{-1}\right).
\end{align}

Notice that, for a fixed $k$, $\lim_{n \rightarrow \infty}\frac{|\underline{\partial} \block_{n}| + |R_{n-1,k}|}{|\block_{n}|} = 0$ and $\lim_{n \rightarrow \infty}\frac{|B_{n-k-1}|}{|\block_{n}|} = 1$. Therefore,
\begin{align}
-\varepsilon c^{-1} + \int \hat{I}_{\pi}(\omega) d\nu	&	\leq    \liminf_{n \rightarrow \infty} \int \frac{-\log \pi^{\omega}_{\Lambda_n}(\omega({\block_{n}}))}{|\block_{n}|} d\nu  \\
                                							&	\leq    \limsup_{n \rightarrow \infty} \int \frac{-\log \pi^{\omega}_{\Lambda_n}(\omega({\block_{n}}))}{|\block_{n}|} d\nu	\\
										&	\leq    \int \hat{I}_{\pi}(\omega) d\nu + \varepsilon c^{-1}.
\end{align}

By letting $\varepsilon \rightarrow 0$, we see that:
\begin{equation}
\lim_{n \rightarrow \infty} \int \frac{-\log \pi^{\omega}_{\Lambda_n}(\omega({\block_{n}}))}{|\block_{n}|} d\nu = \int{\hat{I}_{\pi}(\omega)}d\nu,
\end{equation}
completing the proof.
\end{proof}


\section{Spatial mixing and stochastic dominance}
\label{section7}

From now on, when talking about specifications for the Potts, Widom-Rowlinson and hard-core lattice models, we will distinguish them by the subindex corresponding to the parameter $\beta$, $\lambda$ or $\gamma$ of the model, i.e. $\pi_{\beta,\Lambda}^\xi$ should be understood as a probability measure in the Potts model, $\pi_{\lambda,\Lambda}^\xi$ in the Widom-Rowlinson and $\pi_{\gamma,\Lambda}^\xi$ in the hard-core lattice model, and $\pi_\beta$, $\pi_\lambda$ and $\pi_\gamma$ will denote the corresponding specifications. Also, we will write $\pi^\beta_\Lambda$, $\hat{\pi}^\beta$ and $\hat{I}^\beta_\pi$ for the functions  $\pi_\Lambda$, $\hat{\pi}$ and $\hat{I}_\pi$ in the Potts model, and  short-hand notations when
$\Lambda = S_n$ or $S_{y,z}$.  For example,
$$
\pi^\beta_n(\omega) := \pi^\beta_{S_n}(\omega) := \pi^{\omega(\partial S_n)}_{\beta, S_n}(\theta(\zero) = \omega(
\zero))
$$
The analogous notation will be used for the Widom-Rowlinson and hard-core cases, but using the parameters $\lambda$ and $\gamma$, respectively.

\subsection{Spatial mixing properties}

We now introduce concepts of spatial mixing that we will need in this paper. Let $f:\N \rightarrow \R_{\geq 0}$ be a function such that $f(n) \searrow 0$ as $n \to \infty$.

\begin{definition}
\label{SSMspec}
Given $\Lambda \subseteq \Z^d$, we say that a $\Lambda$-MRF $\mu$ satisfies \emph{strong spatial mixing (SSM) with rate $f(n)$ for a class of finite sets $\mathcal{C}$} if for any $\Delta \in \mathcal{C}$ such that $\overline{\Delta} \subseteq \Lambda$, any $\Theta \subseteq \Delta$, $\theta \in \Symb^\Theta$ and $\xi,\eta \in \Symb^{\partial \Delta}$ with $\mu(\xi)\mu(\eta) > 0$,
\begin{equation}
\left| \mu(\theta \vert \xi) - \mu(\theta \vert \eta) \right| \leq |\Theta|f\left(\dist(\Theta,\Sigma_{\partial \Delta}(\xi,\eta))\right).
\end{equation}

We say that a Gibbs specification $\pi = \{\pi_\Lambda^\xi\}_{\Lambda,\xi}$ satisfies SSM with rate $f(n)$ for a class of finite sets $\mathcal{C}$ if each element $\pi_\Lambda^\xi$ satisfies SSM with rate $f(n)$ for the class $\mathcal{C}$ restricted to subsets of $\Lambda$.

If there exists $C,\alpha > 0$ such that $f$ can be chosen to be $f(n) = Ce^{-\alpha n}$, we say that \emph{exponential SSM} holds.
\end{definition}

\begin{definition}(\cite[p. 445]{1-alexander})
A $\Z^d$-MRF $\mu$ satisfies the \emph{ratio strong mixing property for a class of finite sets $\mathcal{C}$} if there exists $C,\alpha > 0$ such that for any $\Delta \in \mathcal{C}$, any $\Theta, \Sigma \subseteq \Delta$ and $\xi \in \Symb^{\partial \Delta}$ with $\mu(\xi) > 0$,
\begin{align}
\sup\left\{\left|\frac{\mu(A \cap B \vert \xi)}{\mu(A \vert \xi)\mu(B \vert \xi)} - 1\right|: A \in \mathcal{F}_\Theta, B \in \mathcal{F}_\Sigma, \mu(A \vert \xi)\mu(B \vert \xi) > 0\right\} \nonumber\\
\leq C \sum_{x \in \Theta, y \in \Sigma} e^{-\alpha \dist(x,y)}.
\end{align}
\end{definition}

\begin{proposition}
\label{SSMequiv}
Let $\mu$ be a $\Z^2$-MRF with $\supp(\mu) = \Symb^{\Z^2}$ that satisfies the ratio strong mixing property for the class of finite simply lattice-connected sets. Then, $\mu$ satisfies exponential SSM for the family of sets $\{S_{y,z}\}_{y,z \geq 0}$.
\end{proposition}

\begin{proof}
Fix $y,z \geq 0$ and the corresponding set $S_{y,z} \Subset \Z^2$. Let $\Theta \subseteq S_{y,z}$, $\theta \in \Symb^\Theta$ and $\xi_1,\xi_2 \in \Symb^{\partial S_{y,z}}$ with $\mu(\xi_1)\mu(\xi_2) > 0$, consider:
\begin{enumerate}
\item the sets $\Sigma := \Sigma_{\partial S_{y,z}}(\xi_1,\xi_2)$ and $\Delta := S_{y,z} \cup \Sigma$,
\item an arbitrary configuration $\tilde{\xi} \in \Symb^{\partial \Delta}$ such that $\tilde{\xi}(\partial S_{y,z} \setminus \Sigma) = \xi_i(\partial S_{y,z} \setminus \Sigma)$ ($i = 1,2$), and
\item the events $A := [\theta] \in \mathcal{F}_\Theta$ and $B_i := [\xi_i(\Sigma)] \in \mathcal{F}_\Sigma$, for $i = 1,2$.
\end{enumerate}

Notice that $\Delta$ is a finite simply lattice-connected set and, since $\supp(\mu) = \Symb^{\Z^d}$, we can be sure that $\mu(\tilde{\xi}) > 0$. Then:
\begin{align}
\left| \mu(\theta \vert \xi_1) - \mu(\theta \vert \xi_2) \right|	&	=	\left| \mu(A \vert [\tilde{\xi}] \cap B_1) - \mu(A \vert [\tilde{\xi}] \cap B_2) \right|	\\
											&	=	\left| \frac{\mu(A \cap B_1 \vert \tilde{\xi})}{\mu(B_1 \vert \tilde{\xi})} - \mu(A \vert \tilde{\xi}) + \mu(A \vert \tilde{\xi}) - \frac{\mu(A \cap B_2 \vert \tilde{\xi})}{\mu(B_2 \vert \tilde{\xi})} \right|		\\
											&	\leq	\left| \frac{\mu(A \cap B_1 \vert \tilde{\xi})}{\mu(B_1 \vert \tilde{\xi})\mu(A \vert \tilde{\xi})} - 1 \right| + \left| 1 - \frac{\mu(A \cap B_2 \vert \tilde{\xi})}{\mu(B_2 \vert \tilde{\xi})\mu(A \vert \tilde{\xi})} \right|		\\
											&	\leq	2C \sum_{x \in \Theta, y \in \Sigma} e^{-\alpha \dist(x,y)}	\\
											&	\leq	|\Theta| 2C \sum_{y \in \Sigma} e^{-\alpha \dist(\Delta,y)}.
\end{align}

W.l.o.g., we can assume that $|\Sigma| = 1$ (see \cite[Corollary 2]{1-briceno}). Therefore, by taking $C' = 2C$, we have:
\begin{equation}
\left| \mu(\theta \vert \xi_1) - \mu(\theta \vert \xi_2) \right| \leq |\Theta| 2K \sum_{y \in \Sigma} e^{-\alpha \dist(\Delta,y)} = |\Theta| C'e^{-\alpha \dist(\Theta,\Sigma)}.
\end{equation}
\end{proof}

\begin{remark}
The proof of Proposition \ref{SSMequiv} seems to require some assumption on the support of $\mu$ (for the existence of $\tilde{\xi}$ in the enumerated item list above). Fully supported (i.e. $\supp(\mu) = \Symb^{\Z^2}$) suffices, and is the only case in which we will apply this result (see Corollary \ref{potts-ratio}), but the conclusion probably holds under weaker assumptions.
\end{remark}

Given $y,z \geq 0$, we define the \emph{bottom boundary of $S_{y,z}$} as $\bBdry S_{y,z} := \partial S_{y,z} \cap \past$, i.e. the portion of the boundary of $S_{y,z}$ included in the past, and the \emph{top boundary of $S_{y,z}$} as the complement $\tBdry S_{y,z} := \partial S_{y,z} \setminus \past$. Clearly, $\partial S_{y,z} = \bBdry S_{y,z} \sqcup \tBdry S_{y,z}$.

\begin{proposition}
\label{SSM_subcritical}
Let $\pi$ be a specification satisfying exponential SSM with parameters $C,\alpha > 0$. Then, for all $n \in \N$, $y,z \geq \vecone n$ and $a \in \Symb$:
\begin{equation}
\left| \pi_{S_n}^{\omega_1}(\theta(\veczero) = a) - \pi_{S_{y,z}}^{\omega_2}(\theta(\veczero) = a)\right| \leq Ce^{-{\alpha}n},
\end{equation}
uniformly over $\omega_1, \omega_2 \in \Omega(\mathcal{E})$ such that $\omega_1(\past) = \omega_2 (\past)$.
\end{proposition}

\begin{proof}
Fix $n \in \N$, $y,z \geq \vecone n$, $a \in \Symb$ and $\omega_1, \omega_2 \in \Omega(\mathcal{E})$ such $\omega_1(\past) = \omega_2 (\past)$. Denote $\xi := \omega_1(\partial S_n)$. Then:
\begin{align}
	&	\left| \pi_{S_n}^{\omega_1}(\theta(\veczero) = a) - \pi_{S_{y,z}}^{\omega_2}(\theta(\veczero) = a)\right|	\\
=~	&	\left| \pi_{S_n}^{\xi}(\theta(\veczero) = a) - \sum_{\eta}\pi_{S_{y,z}}^{\omega_2}(\theta(\veczero) = a \vert \eta)\pi_{S_{y,z}}^{\omega_2}(\eta)\right| \\
\leq~	& 	\sum_{\eta}\left| \pi_{S_n}^{\xi}(\theta(\veczero) = a) - \pi_{S_n}^{\eta}(\theta(\veczero) = a )\right|\pi_{S_{y,z}}^{\omega_2}(\eta)	\\
\leq~	& 	\sum_{\eta}Ce^{-{\alpha}\dist(\veczero,\Sigma_{\partial S_n}(\xi,\eta))}\pi_{S_{y,z}}^{\omega_2}(\eta) \leq Ce^{-{\alpha}n},
\end{align}
where the summation $\sum_{\eta}$ is taken over all $\eta \in \Symb^{\partial S_n}$ such that $\pi_{S_{y,z}}^{\omega_2}(\eta) > 0$ and $\eta\left(\bBdry S_n\right) = \omega_2\left(\bBdry S_n\right)$. The last inequality above follows from the fact that for any such $\eta$, $\Sigma_{\partial S_n}(\xi,\eta) \subseteq \tBdry S_n$, so:
\begin{equation}
\dist(\veczero,\Sigma_{\partial S_n}(\xi,\eta)) \geq \dist\left(\veczero,\tBdry S_n\right) = n.
\end{equation}
\end{proof}

\begin{definition}[Variational distance]
Let $S$ be a finite set and let $X_1$ and $X_2$ be two $S$-valued random variables with distribution $\rho_1$ and $\rho_2$, respectively. The \emph{variational distance $d_{TV}$} of $X_1$ and $X_2$ (or equivalently, of $\rho_1$ and $\rho_2$) is defined by:
\begin{equation}
d_{TV}(\rho_1,\rho_2) := \frac{1}{2} \sum_{x \in S} \left|\rho_1(x) - \rho_2(x)\right|.
\end{equation}
\end{definition}

It is well-known that $d_{TV}(\rho_1,\rho_2)$ is a lower bound on $\Prob(X_1 \neq X_2)$ over all couplings $\Prob$ of $\rho_1$ and $\rho_2$ and that there is a coupling, called the {\em optimal coupling}, that achieves this lower bound.

Given a Gibbs specification $\pi$, we define:
\begin{equation}
Q(\pi):= \max_{\omega_1,\omega_2 \in \Omega(\mathcal{E})} d_{TV}\left(\pi_{\{\veczero\}}^{\omega_1}(\cdot),\pi_{\{\veczero\}}^{\omega_2}(\cdot)\right).
\end{equation}

The following result is essentially in \cite{1-berg}.

\begin{theorem}
\label{vdBM-SSM}
Let $\pi$ be a Gibbs specification for a n.n. interaction $\Phi$ and a set of constraints $\mathcal{E}$, such that $\Omega(\mathcal{E})$ has a safe symbol. Then, if $p_{\rm c}$ denotes the critical value of site percolation on $\Z^2$ and $Q(\pi) < p_{\rm c}$, we have that $\pi$ satisfies exponential SSM.
\end{theorem}

\begin{proof}
Take $\mu$ any n.n. Gibbs measure for $\Phi$. Since $\Omega(\mathcal{E})$ has a safe symbol, $\mu$ is fully supported, i.e. $\supp(\mu) = \Omega(\mathcal{E})$ (very special case of \cite[Remark 1.14]{1-ruelle}). Given a $\Z^d$-MRF $\mu$, define:
\begin{equation}
Q(\mu):= \max_{\eta_1,\eta_2} d_{TV}(\mu(\theta(\veczero) = \cdot \vert \eta_1), \mu(\theta(\veczero) = \cdot \vert \eta_2)),
\end{equation}
where $\eta_1$ and $\eta_2$ range over all configurations on $\partial \{\veczero\}$ such that $\mu(\eta_1),\mu(\eta_2) > 0$. Then, $Q(\mu) \leq Q(\pi) < p_{\rm c}$, so by \cite[Theorem 1]{1-berg} and shift-invariance of $\Phi$, $\mu$ satisfies exponential SSM (see \cite[Theorem 3.10]{3-marcus}). Finally, since $\mu$ is fully supported, we can conclude that $\pi$ satisfies exponential SSM.
\end{proof}

\subsection{Stochastic dominance}

Suppose that $\Symb$ is a finite linearly ordered set. Then for any set $L$ (in our context, usually a set of sites or bonds), $\Symb^L$ is equipped with a natural partial order $\preceq$ which is defined coordinate-wise: for $\theta_1, \theta_2 \in \Symb^L$, we write $\theta_1 \preceq \theta_2$ if $\theta_1(x) \leq \theta_2(x)$ for every $x \in L$. A function $f: \Symb^L \to \R$ is said to be \emph{increasing} if $f(\theta_1) \leq f(\theta_2)$ whenever $\theta_1 \preceq \theta_2$. An event $A$ is said to be increasing if its characteristic function $\chi_A$ is increasing.

\begin{definition}
Let $\rho_1$ and $\rho_2$ be two probability measures on $\Symb^L$. We say that $\rho_1$ is \emph{stochastically dominated} by $\rho_2$, writing $\rho_1 \leq_D \rho_2$, if for every bounded increasing function $f: \Symb^L \to \R$ we have $\rho_1(f) \leq \rho_2(f)$, where $\rho(f)$ denotes the expected value $\mathbb{E}_{\rho}(f)$ of $f$ according to the measure $\rho$.
\end{definition}

\subsubsection{Stochastic dominance and connectivity decay for the bond random-cluster model}

Recall from Section \ref{bond_cluster} the bond random-cluster model on finite subsets of $\Z^2$ with boundary conditions $i = 0,1$, and the bond random-cluster model $\phi_{p,q}$ on $\Z^2$ (see page \pageref{phi-p-q}).

\begin{theorem}[{\cite[Equation (29)]{2-georgii}}]
\label{stoch_dom_RC}
For any $p \in [0,1]$ and $q \in \N$, and any $\Delta \subseteq \Lambda \Subset \Z^2$:
\begin{equation}
\phi^{(0)}_{p,q,\Delta} \leq_D \phi^{(0)}_{p,q,\Lambda} \mbox{ and } \phi^{(1)}_{p,q,\Lambda} \leq_D \phi^{(1)}_{p,q,\Delta}.
\end{equation}

In particular, if $p < p_{\rm c}(q)$, we have that, for any $\Lambda \Subset \Z^2$:
\begin{equation}
\phi^{(0)}_{p,q,\Lambda} \leq_D \phi_{p,q} \leq_D \phi^{(1)}_{p,q,\Lambda},
\end{equation}
where $\leq_D$ is with respect to the restriction of each measure to events on $E^0(\Lambda)$.
\end{theorem}

The following result was a key element of the proof that $\beta_c(q) = \log(1 + \sqrt{q})$ is the critical inverse temperature for the Potts model.  We will use this result in a crucial way.

Recall that for $p < p_{\rm c}(q)$,  $\phi_{p,q}$ is the unique bond
random cluster measure with parameters $p$ and $q$.

\begin{theorem}[{\cite[Theorem 2]{1-beffara}}]
\label{hugo}
Let $q \geq 1$. For any $p < p_{\rm c}(q) = \frac{\sqrt{q}}{1 + \sqrt{q}}$, the two-point connectivity function decays exponentially, i.e. there exist $0 < C(p,q), c(p,q) < \infty$ such that for any $x,y \in \Z^2$:
\begin{equation}
\phi_{p,q}(x \leftrightarrow y) \leq C(p,q)e^{-c(p,q)\|x-y\|_2},
\end{equation}
where $\{x \leftrightarrow y\}$ is the event that the sites $x$ and $y$ are connected by an open path and $\|\cdot\|_2$ is the Euclidean norm.
\end{theorem}

\subsubsection{Stochastic dominance for the site random-cluster model}

\begin{lemma}
\label{unlikely0}
Given a set $\Lambda \Subset \Z^d$ and parameters $p \in [0,1]$ and $q > 0$, we have that for any $x \in \Lambda$ and any $\tau \in \{0,1\}^{\Lambda \setminus \{x\}}$:
\begin{equation}
p_1(q) \leq \psi^{(1)}_{p,q,\Lambda}(\theta(x) = 1 \vert \tau) \leq p_2(q),
\end{equation}
where $p_1(q) = \frac{pq}{pq + (1-p)q^{2d}}$ and $p_2(q) = \frac{pq}{pq + (1-p)}$. In consequence,
\begin{equation}
\psi_{p_1(q),\Lambda} \leq_D \psi^{(1)}_{p,q,\Lambda} \leq_D \psi_{p_2(q),\Lambda}.
\end{equation}

(Recall that $\Psi_{p,\Lambda}$ denotes Bernoulli site percolation).
\end{lemma}

\begin{proof}
This result is obtained by adapting the discussion on \cite[p. 339]{1-grimmett} to the wired site random-cluster model. See also \cite[Lemma 5.4]{1-higuchi} for the case $q = 2$.
\end{proof}

\subsubsection{Stochastic dominance for the Potts model}
\label{Potts_stoch}

As before, let $q \in \Symb_q$ denote a fixed, but arbitrary, choice of a colour. Let $\Lambda \Subset \Z^d$ and consider $g:\Symb_q^\Lambda \rightarrow \{+,-\}^\Lambda$ be defined by:
\begin{equation}
(g(\theta))(x) = \begin{cases}
+	& \mbox{ if } \theta(x) = q,		\\
-	& \mbox{ if } \theta(x) \neq q.
\end{cases}
\end{equation}

The function $g$ makes the non-$q$ colours indistinguishable and gives a \emph{reduced model}. We say $\theta \simeq \theta'$ if $g(\theta) = g(\theta')$.  This relation defines a partition of $\Symb_q^\Lambda$ and unions of elements of this partition form a sub-algebra of $\Symb_q^\Lambda$, which can be identified with the collection of all subsets of $\{+,-\}^\Lambda$. Let $\pi_{\beta,\Lambda}^+ := g_* \pi_{\beta,\Lambda}^{\omega_q}$ be the push-forward measure, which is nothing more than the restriction (projection) of $\pi_{\beta,\Lambda}^{\omega_q}$ to $\{+,-\}^\Lambda$. Chayes showed that the FKG property holds on events in this reduced model. In particular:

\begin{proposition}[{\cite[Lemma on p. 211]{1-chayes}}]
\label{Potts_FKG}
For all $\beta > 0$ and $\Lambda \Subset \Z^2$, $\pi_{\beta,\Lambda}^+$ satisfies the following properties:
\begin{enumerate}
\item For increasing subsets $A,B \subseteq \{+,-\}^\Lambda$: $\pi_{\beta,\Lambda}^+(A \ | \ B) \geq \pi_{\beta,\Lambda}^+(A)$.
\item If $A$ is decreasing and $B$ is increasing, then: $\pi_{\beta,\Lambda}^+(A \ | \ B) \leq \pi_{\beta,\Lambda}^+(A)$.
\item If $\Delta \subseteq \Lambda$ and $A$ is an increasing subset of $\{+,-\}^\Delta$, then: $\pi_{\beta,\Delta}^+(A) \geq \pi_{\beta,\Lambda}^+(A)$.
\end{enumerate}
\end{proposition}

\begin{proof}
~
\begin{enumerate}
\item This is contained in \cite[Lemma on p. 211]{1-chayes}.
\item This is an immediate consequence of (1).
\item This is a standard consequence of (1): Let $B = +^{\partial \Delta}$. Since $g^{-1}(B)$ is a single configuration, namely $q^{\partial \Delta}$, we obtain from the Markov property of $\pi_{\beta,\Lambda}^{\omega_q}$ that $\pi_{\beta,\Delta}^+(A) = \pi_{\beta,\Lambda}^+(A \ | \ B)$. From (1), we have $\pi_{\beta,\Lambda}^+(A \ | \ B) \geq \pi_{\beta,\Lambda}^+(A)$. Now, combine the previous two statements.
\end{enumerate}
\end{proof}

\begin{remark}
\label{Remark_FKG}
The preceding result immediately applies to $\pi_{\beta,\Lambda}^{\omega_q}$ for events in $\Symb_q^\Lambda$ that are measurable with respect to $\{+,-\}^\Lambda$, viewed as a sub-algebra of $\Symb_q^\Lambda$.
\end{remark}

\subsubsection{Volume monotonicity for the Widom-Rowlinson model with $2$ types}

For the classical Widom-Rowlinson model ($q=2$), Higuchi and Takei showed that the FKG property holds. In particular,

\begin{proposition}[{\cite[Lemma 2.3]{1-higuchi}}]
\label{WR_FKG}
Fix $q = 2$ and let $\Delta \subseteq \Lambda \Subset \Z^d$ and $\lambda > 0$. Then:
\begin{equation}
\pi_{\Lambda}^{\lambda}(\omega_q) \leq \pi_{\Delta}^{\lambda}(\omega_q).
\end{equation}
\end{proposition}

However, this kind of stochastic monotonicity can fail for general $q$ (see \cite[p. 60]{2-georgii}).


\section{Exponential convergence of $\pi_{n}$ in $\Z^2$ lattice models}
\label{section8}

In this section, we consider the Potts, Widom-Rowlinson and hard-core models and establish exponential convergence results that will lead to pressure representation and approximation algorithms  for these lattice models.

Recall that for the Potts model, $\pi^\beta_{y,z}(\omega) =  \pi^{\omega(\partial S_{y,z})}_{\beta, S_n}(\theta(\veczero) = \omega(\veczero))$ and, in particular, $\pi^\beta_n(\omega) =  \pi^{\omega(\partial S_n)}_{\beta, S_n}(\theta(\veczero) = \omega(\veczero))$, with similar notation for the Widom-Rowlinson and hard core models.

\subsection{Exponential convergence in the Potts model}

\begin{theorem}
\label{potts-decay}
For the Potts model with $q$ types and inverse temperature $\beta$, there exists a critical parameter $\beta_{\rm c}(q) > 0$ such that for $0 < \beta \neq \beta_c(q)$, there exists $C,\alpha > 0$ such that, for every $y,z \geq \vecone n$:
\begin{equation}
\label{potts-bound}
\left| \pi^\beta_{n}(\omega_q) - \pi^\beta_{y,z}(\omega_q)\right| \leq Ce^{-{\alpha}n}.
\end{equation}
\end{theorem}

\begin{proof}
In the supercritical region $\beta > \beta_c(q)$, our proof very closely follows ~\cite[Theorem 3]{2-chayes}, which treated the Ising case. We fill in some details of their proof, adapting that proof in two ways:  to a  half-plane version of their result (the quantities in (\ref{potts-bound}) are effectively half-plane quantities) and to the general Potts case. For the subcritical region $\beta < \beta_c(q)$, the proposition will follow easily from~\cite[Theorem 1.8 (ii)]{1-alexander}.

\bigskip
\noindent
{\bf Part I: $\beta>\beta_{\rm c}(q)$.} Let $\Path^{-\star}_{\partial S_n}$ denote the event that there is a $\star$-path of $-$ from $\veczero$ to $\partial S_n$, i.e. a path that runs along ordinary $\Z^2$ bonds and diagonal bonds where the colour at each site is \emph{not} $q$ (in our context below, the configuration on the bottom piece $\bBdry S_{y,z}$ of ${\partial S_n}$ will be all $q$ and thus a $\star$-path of $-$ from $\veczero$ to $\partial S_n$ cannot terminate on $\bBdry S_n$). Note that $\Path^{-\star}_{\partial S_n}$ is an event that is measurable with respect to the sub-algebra $\{+,-\}^{\Lambda}$, for any finite set $\Lambda$ containing $S_n$, introduced in Section \ref{Potts_stoch} (recall that this sub-algebra corresponds to the reduced Potts model).

By decomposing $\pi^{\beta}_{y,z}(\omega_q)$ into probabilities conditional on $\Path^{-\star}_{\partial S_n}$ and $(\Path^{-\star}_{\partial S_n})^c$ , we obtain:
\begin{align}
\label{Eq1}
	&	\pi^{\beta}_n(\omega_q) - \pi^{\beta}_{y,z}(\omega_q)	 \\
=~	&	\pi^{\omega_q}_{\beta,S_n}(\theta(\veczero) = q) - \pi^{\omega_q}_{\beta,S_{y,z}}(\theta(\veczero) = q)	\\
=~	&	\pi^{\omega_q}_{\beta,S_{y,z}}(\Path^{-\star}_{\partial S_n})\left(\pi^{\omega_q}_{\beta,S_n}(\theta(\veczero) = q) - \pi^{\omega_q}_{\beta,S_{y,z}}(\theta(\veczero) = q \vert \Path^{-\star}_{\partial S_n})\right)  \nonumber \\
	&	+ (1-\pi^{\omega_q}_{\beta,S_{y,z}}(\Path^{-\star}_{\partial S_n}))\left(\pi^{\omega_q}_{\beta,S_n}(\theta(\veczero) = q) - \pi^{\omega_q}_{\beta,S_{y,z}}(\theta(\veczero) = q \vert (\Path^{-\star}_{\partial S_n})^{\rm c})\right) \label{Eq5}.
\end{align}

We claim that the expression in (\ref{Eq1}) is nonnegative. To see this, observe that the events $\{\omega(\veczero) = q\}$, $\{\omega(\partial S_n) = q^{\partial S_n}\}$ and
 $\{\omega(\partial S_{y,z}) = q^{\partial S_{y,z}}\}$
may be viewed as the events $\{\omega(\veczero) = +\}$, $\{\omega(\partial S_n) = +^{\partial S_n}\}$
and $\{\omega(\partial S_n) = +^{\partial S_{y,z}}\}$ in the sub-algebra $\{+,-\}^{S_{y,z}}$ of the reduced model, as discussed in Section \ref{Potts_stoch}. Now, apply Proposition \ref{Potts_FKG} (part 3) and Remark \ref{Remark_FKG}.

We next claim that:
\begin{equation}
\label{decompose}
\pi^{\omega_q}_{\beta,S_{y,z}}(\theta(\veczero) = q \vert (\Path^{-\star}_{\partial S_n})^{\rm c}) \geq \pi^{\omega_q}_{\beta,S_n}(\theta(\veczero) = q).
\end{equation}

To be precise, first observe that  $\omega \in (\Path^{-\star}_{\partial S_n})^c$ iff $\omega$ contains an all-$q$ path in $S_n$ from $\partial \past \cap \{x_1 < 0\}$ to $\partial \past \cap \{x_1 > 0\}$. So, $(\Path^{-\star}_{\partial S_n})^c$ can be decomposed into a disjoint collection of events determined by the unique furthest such path from $\veczero$.  Using the MRF property of Gibbs measures, it follows that we can regard each of these events as an increasing event in  $\{+,-\}^{S_m}$. Now, apply Proposition \ref{Potts_FKG} and Remark \ref{Remark_FKG}. (The reader may notice that here we have essentially used the strong Markov property (see \cite[p. 1154]{4-georgii}).)

Thus, (\ref{Eq5}) is nonpositive. This, together with the fact that $\pi^{\omega_q}_{\beta,S_n}(\theta(\veczero) = q \vert \Path^{-\star}_{\partial S_n}) = 0$, yields:
\begin{equation}
0 \leq \pi^{\beta}_n(\omega_q) - \pi^{\beta}_{y,z}(\omega_q) \leq \pi^{\omega_q}_{\beta,S_{y,z}}(\Path^{-\star}_{\partial S_n})\pi^{\omega_q}_{\beta,S_n}(\theta(\veczero) = q) \leq \pi^{\omega_q}_{\beta,S_{y,z}}(\Path^{-\star}_{\partial S_n}).
\end{equation}

So, it suffices to show that $\sup_{y,z \geq \vecone n} \pi^{\omega_q}_{\beta,S_{y,z}}(\Path^{-\star}_{\partial S_n})$ decays exponentially in $n$. Fix $y,z \geq \vecone n$ and let $m > n$ such that $\vecone m \geq y,z$. By Proposition \ref{Potts_FKG} (parts 2 and 3) and Remark \ref{Remark_FKG},
\begin{align}
\label{SnBn}
\pi_{\beta,S_{y,z}}^{\omega_q}(\Path^{-\star}_{\partial S_n})	&	\leq \pi_{\beta,S_m}^{\omega_q}(\Path^{-\star}_{\partial S_n})	\\
												&	= \pi_{\beta,\block_m}^{\omega_q}(\Path^{-\star}_{\partial S_n} \vert q^\past) \leq \pi_{\beta,\block_m}^{\omega_q}(\Path^{-\star}_{\partial S_n}) 	\leq	\pi_{\beta,\block_m}^{\omega_q}(\Path^{-\star}_{\partial \block_n}).
\end{align}

So, it suffices to show that $\sup_{m > n}\pi^{\omega_q}_{\beta,\block_m}(\Path^{-\star}_{\partial \block_n})$ decays exponentially in $n$. Recall the Edwards-Sokal coupling $\Prob^{(1)}_{p,q,\block_m}$ for the Gibbs distribution and the corresponding bond random-cluster measure with wired boundary condition $ \phi_{p,q,\block_m}^{(1)}$ (see Section~\ref{bond_cluster}).

W.l.o.g., let's suppose that $n$ is even, i.e. $n = 2k < m$, for some $k \in \N$. We consider the following two events in the bond random-cluster model, as in \cite[Theorem 3]{1-chayes}. Let $\Rscr_n$ be the event of an open circuit in $\block_{2k} \setminus \block_{k}$ that surrounds $\block_{k}$. Let $\Mscr_{n,m}$ be the event in which there is an open path from some site in $\block_{k}$ to $\partial \block_m$. The joint occurrence of these two events forces the Potts event $(\Path_{\partial\block_n}^{-\star})^{\rm c}$ in the coupling: $\Rscr_n \cap \Mscr_{n,m} \subseteq (\Path_{\partial\block_n}^{-\star})^c$ (here, technically, we are identifying these events with their inverse images of the projections in the coupling).

Then, by the coupling property:
\begin{align}
\pi_{\beta,\block_m}^{\omega_q}\left((\Path_{\partial\block_n}^{-\star})^c\right)	&   =		\Prob^{(1)}_{p,q,\block_m}\left((\Path_{\partial\block_n}^{-\star})^c\right)   \\
                                           										&   \geq    \Prob^{(1)}_{p,q,\block_m}\left((\Path_{\partial\block_n}^{-\star})^c\middle\vert \Rscr_n\cap\Mscr_{n,m}\right)\Prob^{(1)}_{p,q,\block_m}\left(\Rscr_n\cap\Mscr_{n,m}\right)  \\
                                                										&   =  	\phi^{(1)}_{p,q,\block_m}\left(\Rscr_n \cap \Mscr_{n,m}\right),
\end{align}
so:
\begin{equation}
\pi_{\beta,\block_m}^{\omega_q}(\Path_{\partial\block_n}^{-\star})\le 1-\phi^{(1)}_{p,q,\block_m}(\Rscr_n\cap\Mscr_{n,m})\le \phi^{(1)}_{p,q,\block_m}(\Rscr_n^{\rm c})+  \phi^{(1)}_{p,q,\block_m}(\Mscr_{n,m}^{\rm c}).
 \end{equation}

Therefore,
\begin{equation}
\label{bound1}
\sup_{m > n} \pi_{\beta,\block_m}^{\omega_q}(\Path_{\partial\block_n}^{-\star}) \leq \sup_{m > n} \phi^{(1)}_{p,q,\block_m}(\Rscr_n^{\rm c}) + \sup_{m > n}\phi^{(1)}_{p,q,\block_m}(\Mscr_{n,m}^{\rm c}).
 \end{equation}

The first term on the right hand side of \eqref{bound1} is bounded from above as follows:
\begin{align}
\phi^{(1)}_{p,q,\block_{m}}(\Rscr_n^{\rm c})	&	\leq	\phi^{(1)}_{p,q,\tilde{\block}_{m+1}}(\Rscr_n^{\rm c})	\\
									&	\leq \sum_{x \in \partial \block_{k}, y \in \underline{\partial} \block_{2k}} \phi^{(0)}_{p^*,q,\block_{m+1}}(x \leftrightarrow y)	\\
									&	\leq \sum_{x \in \partial \block_{k}, y \in \underline{\partial} \block_{2k}} \phi_{p^*,q}(x \leftrightarrow y),
\end{align}
where $\tilde{\block}_m = [-m+1,m]^2 \cap \Z^2$ and $p^*$ denotes the dual of $p$ and the inequalities follow from Proposition \ref{RC_dual_prob} and Theorem \ref{stoch_dom_RC}.

If $p > p_{\rm c}(q)$, then $p^* < p_{\rm c} (q)$, and by Theorem \ref{hugo}, the first term on the right side of \eqref{bound1} is upper bounded by $64C(p^*,q)n^2\exp(-c(p^*,q)n/4)$,  since $|\partial \block_{k}||\underline{\partial} \block_{2k}| \leq 64n^2$ and $\|x-y\|_2 \geq k-1 \geq \frac{n}{4}$, for all $x \in \partial \block_{k}$ and $y \in \underline{\partial} \block_{2k}$. So, the first term on the right side of \eqref{bound1} decays exponentially.

As for the second term, in order for $\Mscr_{n,m}$ to fail to occur, there must be a closed circuit in $\block_m \setminus \block_{k}$ and in particular a closed path from $L_{m,n} := \block_m \setminus \block_{k} \cap \{x_1 < 0, x_2 = 0\}$ to $R_{m,n} := \block_m \setminus \block_{k} \cap \{x_1 > 0, x_2 = 0\}$ in $\block_m$. Thus,
\begin{align}
\phi^{(1)}_{p,q,\block_m}(\Mscr_{n,m}^{\rm c})	&	\leq \phi^{(1)}_{p,q,\tilde{\block}_{m+1}}(\Mscr_{n,m}^{\rm c})	\\
									&	\leq \sum_{x \in L_{m,n}, y \in R_{m,n}}  \phi^{(0)}_{p^*,q,\block_{m+1}} (x \leftrightarrow y)	\\
									&	\leq \sum_{x \in L_{m,n}, y \in R_{m,n}}  \phi_{p^*,q}(x \leftrightarrow y),
\end{align}
where the last inequality follows by Proposition \ref{RC_dual_prob} and Proposition \ref{stoch_dom_RC}. By Theorem \ref{hugo}, this is less than:
\begin{align}
\sum_{i = n, j = n} C(p^*,q)e^{-c(p^*,q)(i+j)}	&	\leq C(p^*,q)\left(e^{-c(p^*,q)n}\frac{1}{1-e^{-c(p^*,q)}}\right)^2	\\
									&	= \frac{C(p^*,q)}{(1-e^{-c(p^*,q)})^2}e^{-2c(p^*,q)n}.
\end{align}

Thus, the 2nd term on the right side of \eqref{bound1} decays exponentially, so $\sup_{m > n} \pi^{\omega_q}_{\beta,m}(\Path^{-\star}_{\partial B_n})$ decays exponentially in $n$. Thus, by (\ref{SnBn}) $\sup_{m > n} \pi^{\omega_q}_{\beta,m}(\Path^{-\star}_{\partial S_n})$ also decays exponentially in $n$, as desired.

\bigskip
\noindent
{\bf Part II: $\beta<\beta_{\rm c}(q)$.} Recall from Section \ref{section7} the notions of strong spatial mixing and ratio strong mixing property.

\begin{theorem}[{\cite[Theorem 1.8 (ii)]{1-alexander}}]
\label{Alexander2004}
For the $\Z^2$ Potts model with $q$ types and inverse temperature $\beta$, if $0 < \beta < \beta_{\rm c}(q)$ and exponential decay of the two-point connectivity function holds for the corresponding random-cluster model, then the (unique) Potts Gibbs measure satisfies the ratio strong mixing property for the class of finite simply lattice-connected sets.
\end{theorem}

\begin{corollary}
\label{potts-ratio}
For the $\Z^2$ Potts model with $q$ types and inverse temperature $0 < \beta < \beta_{\rm c}(q)$, the specification $\pi^{\mathrm{FP}}_\beta$ satisfies exponential SSM for the family of sets $\{S_{y,z}\}_{y,z \geq 0}$.
\end{corollary}

\begin{proof}
This follows immediately from Theorem \ref{hugo}, Theorem \ref{Alexander2004} and Proposition \ref{SSMequiv}.
\end{proof}

Then, since exponential SSM holds for the class of finite simply lattice-connected sets when $\beta < \beta_{\rm c}(q)$, the desired result follows directly from Proposition \ref{SSM_subcritical}.

This completes the proof of Theorem \ref{potts-decay}. 
\end{proof}

\subsection{Exponential convergence in the Widom-Rowlinson model}

Recall that for Bernoulli site percolation in $\Z^2$ there exists a probability parameter $p_{\rm c}$, known as the \emph{percolation threshold}, such that for $p < p_{\rm c}$, there is no infinite cluster of $1$'s $\psi_{p,\Z^2}$-almost surely and for $p > p_{\rm c}$, there is such a cluster $\psi_{p,\Z^2}$-almost surely. Similarly, one can define an analogous parameter $p^\star_{\rm c}$ for the lattice $\Z^{2,\star}$, which satisfies $p_{\rm c} + p^\star_{\rm c} = 1$ (see \cite{1-russo}).

\begin{theorem}
\label{widom-decay}
For the Widom-Rowlinson model with $q$ types and activity $\lambda$, there exist two critical parameters $0 < \lambda_1(q) < \lambda_2(q)$ such that for $\lambda < \lambda_1(q)$ or $\lambda > \lambda_2(q)$, there exists $C,\alpha > 0$ such that, for every $y,z \geq \vecone n$:
\begin{equation}
\label{widom-bound}
\left| \pi^\lambda_{n}(\omega_q) - \pi^\lambda_{y,z}(\omega_q)\right| \leq Ce^{-{\alpha}n}.
\end{equation}
\end{theorem}

\begin{proof}

As in the proof of Theorem \ref{potts-decay}, we split the proof in two parts.

\bigskip
\noindent
{\bf Part I: $\lambda > \lambda_2(q): = q^3\left(\frac{p_{\rm c}}{1-p_{\rm c}}\right)$.} Fix $n \in \N$ and $y,z \geq \vecone n$. Notice that, due to the constraints of the Widom-Rowlinson model, and recalling Proposition \ref{WR_cluster}:
\begin{equation}
\pi^\lambda_{y,z}(\omega_q) = \pi_{\lambda, S_{y,z}}^{\omega_q}(\theta(\veczero) = q) = \psi^{(1)}_{p,q,S_{y,z}}(\theta(\veczero) = 1),
\end{equation}
where $p = \frac{\lambda}{1+\lambda}$, and the same holds for $\pi^\lambda_{n}(\omega_q)$. Then, it suffices to prove that:
\begin{equation}
\label{widom-bound}
\left| \psi^{(1)}_{p,q,S_n}(\theta(\veczero) = 1) - \psi^{(1)}_{p,q,S_{y,z}}(\theta(\veczero) = 1)\right| \leq Ce^{-{\alpha}n},
\end{equation}
for some $C,\alpha > 0$.

Notice that $\veczero \in S_n \subseteq S_{y,z} =: \Lambda$. Fix any ordering on the set $\overline{\Lambda}$. From now on, when we talk about comparing sites in $\overline{\Lambda}$, it is assumed we are speaking of this ordering. For convenience, we will extend configurations on $S_n$ and $\Lambda$ to configurations on $\overline{\Lambda}$ by appending $1^{\overline{\Lambda} \setminus S_n}$ and $1^{\partial \Lambda}$, respectively.

Now, we will proceed to define a coupling $\Prob_{n,y,z}$ of $\psi^{(1)}_{p,q,S_n}$ and $\psi^{(1)}_{p,q,\Lambda}$, defined on pairs of configurations $(\theta_1, \theta_2) \in \{0,1\}^{\overline{\Lambda}} \times \{0,1\}^{\overline{\Lambda}}$.  The coupling is defined one site at a time, using values from previously defined sites.

We use $(\tau^t_1,\tau^t_2)$ to denote the (incomplete) configurations on $\overline{\Lambda} \times \overline{\Lambda}$ at step $t = 0,1,\dots,|S_n|$. We therefore begin with $\tau^0_1 = 1^{\overline{\Lambda} \setminus S_n}$ and $\tau^0_2 = 1^{\partial \Lambda}$. Next, we
 set $\tau^1_1  = \tau^0_1$ and form $\tau^1_2$ by extending $\tau^0_2$ to $\overline{\Lambda} \setminus S_n$, choosing randomly according to the distribution $\psi^{(1)}_{p,q,\Lambda}\left(\cdot \middle\vert 1^{\partial \Lambda}\right)$. At this point of the construction, both $\tau^1_1$ and $\tau^1_2$ have shape $\overline{\Lambda} \setminus S_n$. In the end, $(\tau^{|S_n|}_1,\tau^{|S_n|}_2)$ will give as a result a pair $(\theta_1, \theta_2)$.

At any step $t$, we use $W^t$ to denote the set of sites in $\overline{\Lambda}$ on which $\tau^t_1$ and $\tau^t_2$ have already received values in previous steps. In particular, $W^1 = \overline{\Lambda} \setminus S_n$. At an arbitrary step $t$ of the construction, we choose the next site $x^{t+1}$ on which to assign values in $\tau^{t+1}_1$ and $\tau^{t+1}_2$ as follows:
\begin{enumerate}
\item[(i)] If possible, take $x^{t+1}$ to be the smallest site in $\partial^\star W^t$ that is $\star$-adjacent to a site $y \in W^t$ for which $(\tau^t_1(y),\tau^t_2(y)) \neq (1,1)$.
\item[(ii)] Otherwise, just take $x^{t+1}$ to be the smallest site in $\partial^\star W^t$.
\end{enumerate}

Notice that at any step $t$, $W^t$ is a $\star$-connected set, and that it it always possible to find the next site $x^{t+1}$ for any $t < |S_n|$ (i.e. the two rules above give a well defined procedure).

Now we are ready to augment the coupling from $W^t$ to $W^t \cup \{x^{t+1}\}$ by assigning $\tau^{t+1}_1(x^{t+1})$ and $\tau^{t+1}_2(x^{t+1})$ according to an optimal coupling of $\left.\psi^{(1)}_{p,q,S_n}\left(\cdot \middle\vert \tau^t_1\right)\right|_{\{x^{t+1}\}}$ and $\left.\psi^{(1)}_{p,q,S_{y,z}}( \cdot \ | \ \tau^t_2)\right|_{\{x^{t+1}\}}$, i.e. a coupling which minimizes the probability that, given $(\tau^t_1,\tau^t_2)$, $\theta_1(x^{t+1}) \neq \theta_2(x^{t+1})$. Since $\Prob_{n,y,z}$ is defined site-wise, and at each step is assigned according to $\psi^{(1)}_{p,q,S_n}\left(\cdot \middle\vert \tau^t_1\right)$ in the first coordinate and $\psi^{(1)}_{p,q,S_{y,z}}( \cdot \ | \ \tau^t_2)$ in the second, the reader may check that it is indeed a coupling of $\psi^{(1)}_{p,q,S_n}$ and $\psi^{(1)}_{p,q,S_{y,z}}$. The key property of $\Prob_{n,y,z}$ is the following.

\begin{lemma}
\label{path}
$\theta_1(\veczero) \neq \theta_2(\veczero)$ $\Prob_{n,y,z}$-a.s. if and only if there exists a path $\Path$ of $\star$-adjacent sites from $\veczero$ to $\partial S_n$, such that for each site $y \in \Path$, $(\theta_1(y),\theta_2(y)) \neq (1,1)$.
\end{lemma}

\begin{proof}
Suppose, for a contradiction, that $\theta_1(\veczero) \neq \theta_2(\veczero)$ and there exists no such path. This implies that there exists a circuit $\Circ$ surrounding $\veczero$ (when we include the bottom boundary as part of $\Circ$) and contained in $\overline{S}_n$ such that for all $y \in \Circ$, $(\theta_1(y),\theta_2(y)) = (1,1)$. Define by $I$ the simply lattice-$\star$-connected set of sites in the interior of $\Circ$ and, let's say that at time $t_0$, $x^{t_0}$ was the first site within $I$ defined according to the site-by-site evolution of $\Prob_{n,y,z}$. Then, $(\tau^{t_0}_1(x^{t_0}),\tau^{t_0}_2(x^{t_0}))$ cannot have been defined according to rule (i) since all sites $\star$-adjacent to $x^{t_0}$ are either in $I$ (and therefore not yet defined by definition of $x^{t_0}$), or on $\Circ$ (and therefore either not yet defined or sites at which $\theta_1$ and $\theta_2$ are both $1$).

Therefore, $(\theta_1(x^{t_0}),\theta_2(x^{t_0}))$ was defined according to rule (ii). We therefore define the set $D := \overline{\Lambda} \setminus W^{t_0 - 1} \supseteq I$, and note that $\veczero$ and $x^{t_0}$ belong to the same $\star$-connected component $\Theta$ of $D$. We also know that $\tau^{t_0 - 1}_1(\partial^\star D) = \tau^{t_0 - 1}_2(\partial^\star D) = 1^{\partial^\star D}$, otherwise some unassigned site in $D$ would be $\star$-adjacent to a $0$ in either $\tau^{t_0 - 1}_1(\partial^\star D)$ or $\tau^{t_0 - 1}_2(\partial^\star D)$, and so rule (i) would be applied instead. We may now apply Lemma \ref{pseudoMRF} (combined with Remark \ref{notePseudoMRF}) to $\Theta$ and $\Lambda$ in order to see that $\psi^{(1)}_{p,q,S_n}(\theta_1(\Theta) \vert \tau^{t_0 - 1}_1)$ and $\psi^{(1)}_{p,q,S_{y,z}}(\theta_2(\Theta) \vert \tau^{t_0 - 1}_2)$ are identical. This means that the optimal coupling according to which $\tau^{t_0}_1(x^{t_0})$ and $\tau^{t_0}_2(x^{t_0})$ are assigned is supported on the diagonal, and so $\tau^{t_0}_1(x^{t_0}) = \tau^{t_0}_2(x^{t_0})$, $\Prob_{n,y,z}$-almost surely. This will not change the conditions under which we applied Lemma \ref{pseudoMRF}, and so inductively, the same will be true for each site in $I$ as it is assigned, including $\veczero$. We have shown that $\theta_1(\veczero) = \theta_2(\veczero)$, $\Prob_{n,y,z}$-almost surely, regardless of when $\veczero$ is assigned in the site-by-site evolution of $\Prob_{n,y,z}$. This is a contradiction, and so our original assumption was incorrect, implying that the desired path $\Path$ exists.
\end{proof}

Given an arbitrary time $t$, let:
\begin{equation}
\begin{array}{ccc}
 \rho^t_1(\cdot) := \left.\psi^{(1)}_{p,q,S_n}(\cdot \vert \tau^{t-1}_i)\right|_{\{x^t\}} & \mbox{ and }	&	\rho^t_2(\cdot) := \left.\psi^{(1)}_{p,q,\Lambda}(\cdot \vert \tau^{t-1}_i)\right|_{\{x^t\}}
 \end{array}
 \end{equation}
be the two corresponding probability measures defined on the set $\{0,1\}^{\{x^t\}}$. Note that at any step within the site-by-site definition of $\Prob_{n,y,z}$, Lemma \ref{unlikely0} implies that $\frac{\lambda}{\lambda + q^3} \leq \rho^t_i(1)$, where $\lambda = \frac{p}{1-p}$ and $i = 1,2$. Now, w.l.o.g., suppose that $\rho^t_2(0) \geq \rho^t_1(0)$. Then, an optimal coupling $\mathbb{Q}^t$ of $\rho^t_1$ and $\rho^t_2$ will assign $\mathbb{Q}^t(\{(0,0)\}) = \rho^t_1(0)$, $\mathbb{Q}^t(\{(0,1)\}) = 0$, $\mathbb{Q}^t(\{(1,0)\}) = \rho^t_2(0)-\rho^t_1(0)$, and $\mathbb{Q}^t(\{(1,1)\}) = 1-\rho^t_2(0)$. Therefore,
\begin{equation}
\label{bound}
\mathbb{Q}^t(\{(1,1)\}^{\rm c}) = \rho^t_2(0) \leq \frac{q^3}{\lambda+q^3}.
\end{equation}

Next, define the map $h: \{0,1\}^{S_n} \times \{0,1\}^{S_n} \to \{0,1\}^{S_n}$ given by:
\begin{equation}
(h(\theta_1,\theta_2))(x) =
\begin{cases}
1	&	\mbox{if }	(\theta_1(x),\theta_2(x)) \neq (1,1),	\\
0	&	\mbox{if }	(\theta_1(x),\theta_2(x)) = (1,1).
\end{cases}
\end{equation}

By (\ref{bound}), $h_*\Prob_{n,y,z}$ (the push-forward measure) can be coupled against an i.i.d. measure on $\{0,1\}^{S_n}$ which assigns $1$ with probability $\frac{q^3}{\lambda+q^3}$ and $0$ with probability $\frac{\lambda}{\lambda+q^3}$, and that the former is stochastically dominated by the latter. This, together with Lemma \ref{path}, yields
\begin{align}
\left| \psi^{(1)}_{p,q,S_n}(\theta(\veczero) = 1) - \psi^{(1)}_{p,q,S_{y,z}}(\theta(\veczero) = 1)\right|	&	\leq \Prob_{n,y,z}(\theta_1(\veczero) \neq \theta_2(\veczero))	\\
																		&	\leq  \psi_{\frac{q^3}{\lambda+q^3},S_n}(\veczero \overset{\star}{\leftrightarrow} \partial S_n) \label{eqnDecay},
\end{align}

Since we have assumed $\lambda > q^3\left(\frac{p_{\rm c}}{1-p_{\rm c}}\right)$ and $p_{\rm c} + p_{\rm c}^\star = 1$, we have $\frac{q^3}{\lambda + q^3} < p_{\rm c}^\star$. It follows by \cite{2-aizenman,1-menshikov} that the expression in (\ref{eqnDecay}) decays exponentially in $n$. This completes the proof.

\bigskip
\noindent
\textbf{Part II: $\lambda < \lambda_1(q) := \frac{1}{q}\left(\frac{p_{\rm c}}{1-p_{\rm c}}\right)$.} Observe that, by virtue of Proposition \ref{SSM_subcritical}, it suffices to prove that $\pi^{\mathrm{WR}}_\lambda$ satisfies exponential SSM. For this, we use  Theorem \ref{vdBM-SSM}. By considering all cases of nearest-neighbour configurations at the origin, one can compute:
\begin{equation}
Q(\pi_\lambda^{\mathrm{WR}}) = \max_{\omega_1,\omega_2 \in \Omega(\mathcal{E})} d_{TV}(\pi^{\omega_1}_{\lambda,\{\veczero\}}, \pi^{\omega_2}_{\lambda,\{\veczero\}}) = \frac{q\lambda}{1+q\lambda}.
\end{equation}

By Theorem \ref{vdBM-SSM}, we obtain exponential SSM when:
\begin{equation}
\lambda < \frac{1}{q}\left(\frac{p_c}{1-p_c}\right) = \lambda_1(q).
\end{equation}

Uniqueness of Gibbs states in this same region was mentioned in \cite[p. 40]{3-georgii}, by appealing to \cite[Theorem 1]{1-berg} (which is the crux of Theorem \ref{vdBM-SSM}).
\end{proof}

\begin{remark}
In the case $q=2$, it is possible to give an alternative proof of Theorem \ref{widom-decay}, Part I, using the framework of the proof of Theorem \ref{potts-decay}, Part I. The arguments through (\ref{SnBn}) go through, with an appropriate re-definition of events and use of Proposition \ref{WR_FKG} for stochastic dominance. One can then apply Lemma \ref{unlikely0} to give estimates based on the site random-cluster model. (In contrast to Theorem \ref{potts-decay}, Part I, this does not require the use of planar duality). So far, this approach is limited to $q=2$ because we do not know appropriate versions of Proposition \ref{WR_FKG} for $q > 2$.
\end{remark} 

\subsection{Exponential convergence in the hard-core model}

Our argument again relies on proving exponential convergence for conditional measures with respect to certain ``extremal'' boundaries on $S_n$, but these now will consist of alternating $0$ and $1$ symbols rather than a single symbol (recall from Section \ref{hardcore_sec} that $\omega^{(o)}$ is defined as the configuration of $1$'s on all even sites and $0$ on all odd sites).

\begin{theorem}
\label{hard-decay}
For the $\mathbb{Z}^2$ hard-core model with activity $\gamma$, there exist two critical parameters $0 < \gamma_1 < \gamma_2$ such that for any $0 < \gamma < \gamma_1$ or $\gamma > \gamma_2$, there exist $C, \alpha > 0$ such that for every $y,z \geq \vecone n$,
\begin{equation}
\label{hard-bound}
\left|\pi^\gamma_n(\omega^{(o)}) - \pi^\gamma_{y,z}(\omega^{(o)})\right| \leq Ce^{-\alpha n}.
\end{equation}
\end{theorem}

\begin{proof}

As in the previous two theorems, we consider two cases.

\bigskip
\noindent
\textbf{Part I: $\gamma > \gamma_2 := \gammatwo$.} Our proof essentially combines the disagreement percolation techniques of \cite{1-berg} and the proof of non-uniqueness of equilibrium state for the hard-core model due to Dobrushin (see \cite{1-dobrushin}). We need enough details not technically contained in either proof that we present a mostly self-contained argument here. From \cite[Theorem 1]{1-berg} and an averaging argument (as in the proof of Proposition \ref{SSM_subcritical}) on $\tBdry S_n$ induced by a boundary condition on $S_{y,z}$, we know that for any $y,z \geq \vecone n$,
\begin{equation}
\left| \pi^{\gamma}_{n}(\omega^{(o)}) - \pi^{\gamma}_{y,z}(\omega^{(o)}) \right| \leq \Prob_{n,y,z} \left(\exists \textrm{ a path of disagreement from } \veczero \leftrightarrow \tBdry S_n \right)
\end{equation}
for a certain coupling $\Prob_{n,y,z}$ of $\pi_{\gamma,S_n}^{\omega^{(o)}}$ and $\left.\pi_{\gamma,S_{y,z}}^{\omega^{(o)}}\right|_{S_n}$. We do not need the structure of $\Prob_{n,y,z}$ here, but instead note the following: a path of disagreement for the boundaries $\omega^{(o)}(\partial S_n)$ and $\omega^{(o)}(\partial S_{y,z})$ implies that in one of the configurations, all entries on the path will be ``out of phase'' with respect to $\omega^{(o)}$, i.e. that all entries along the path will have $1$ at every odd site and $0$ at every even site rather than the opposite alternating pattern of $\omega^{(o)}$. Then, if we denote by $\mathcal{T}_n$ the event that there is a path $\Path$ from $\veczero \leftrightarrow \tBdry S_n$ with $1$ at every odd site and $0$ at every even site, it is clear that:
\begin{equation}
\Prob_{n,y,z} \left(\exists \textrm{ a path of disagreement from } \veczero \leftrightarrow \tBdry S_n \right) \leq \pi_{\gamma,S_n}^{\omega^{(o)}}(\mathcal{T}_n) + \pi_{\gamma,S_{y,z}}^{\omega^{(o)}}(\mathcal{T}_n).
\end{equation}

Since $y,z \geq \vecone n$ are arbitrary (in particular, $y$ and $z$ can be chosen to be $\vecone n$), it suffices to prove that $\sup_{y,z \geq \vecone n}\pi_{\gamma,S_{y,z}}^{\omega^{(o)}}(\mathcal{T}_n)$ decays exponentially with $n$. Define the set:
\begin{equation}
\Theta_{y,z} = \{\theta \in \{0,1\}^{\overline{S}^\star_{y,z}}: \theta \mathrm{~is~feasible~and~}\theta(\partial^\star S_{y,z}) = \omega^{(o)}(\partial^\star S_{y,z})\}.
\end{equation} 

For any $\theta \in \Theta_{y,z}$, we define $\Sigma_{\veczero}(\theta)$ to be the connected component of $\Sigma_{S_{y,z}}(\theta,\omega^{(o)})$ ($= \{x \in S_{y,z}: \theta(x) \neq \omega^{(o)}(x)\}$) containing the origin $\veczero$. Since $\mathcal{T}_n \subseteq \{\Sigma_{\veczero}(\theta) \cap \tBdry S_n \neq \emptyset\}$, our proof will then be complete if we can show that there exist $C, \alpha > 0$ so that for any $n$ and $y,z \geq \vecone n$, the following holds:
\begin{equation}
\label{outofphase}
\pi_{\gamma,S_{y,z}}^{\omega^{(o)}}(\Sigma_{\veczero}(\theta) \cap \tBdry S_n \neq \emptyset) \leq Ce^{-\alpha n}.
\end{equation}

To prove this, we use a Peierls argument, similar to \cite{1-dobrushin}. 

Fix any $y,z \geq \vecone n$ and for any $\theta \in \Theta_{y,z}$, define $\Sigma_{\veczero}(\theta)$ as above, and let $K(\theta)$ to be the connected component of $\{x \in \overline{S}^\star_{y,z}: \theta(x) = \omega^{(o)}(x)\}$ containing $\partial^{\star} S_{y,z}$. Clearly, $\Sigma_{\veczero}(\theta)$ and $K(\theta)$ are disjoint, $K(\theta) \neq \emptyset$ and, provided $\theta(\veczero) = 0$, $\Sigma_{\veczero}(\theta) \neq \emptyset$. Then, define $\Gamma(\theta) :=  \Sigma_{\veczero}(\theta) \cap \partial K(\theta) \subseteq S_{y,z}$. We note that for any $\theta \in \Theta_{y,z}$ with $\theta(\veczero) = 0$, we have that $\theta(\Gamma(\theta)) = 0^{\Gamma(\theta)}$, since adjacent sites in $ \Sigma_{\veczero}(\theta)$ and $K(\theta)$ must have the same letter by definition of $\Sigma_{\veczero}(\theta)$, and adjacent $1$ symbols are forbidden in the hard-core model. Therefore, every $x \in \Gamma(\theta)$ is even.

We need the concept of \emph{inner external boundary} for a connected set $\Sigma \Subset \Z^2$. The inner external boundary of $\Sigma$ is defined to be the inner boundary of the simply lattice-connected set consisting of the union of $\Sigma$ and the union of all the finite components of $\Z^2 \setminus \Sigma$. Intuitively, the inner external boundary of $\Sigma$ is the inner boundary of the set $\Sigma$ obtained after ``filling in the holes'' of $\Sigma$. Notice that the set $\Gamma(\theta)$ corresponds exactly to the inner external boundary of $\Sigma_{\veczero}(\theta)$. In addition, by \cite[Lemma 2.1 (i)]{1-deuschel}, we know that the inner external boundary of a finite connected set (more generally a finite $\star$-connected set) is $\star$-connected. Thus, $\Gamma(\theta) \subseteq S_{y,z}$ is a $\star$-connected set $\Circ^\star$ that consists only of even sites and contains the origin $\veczero$, for any $\theta \in \Theta_{y,z}$ with $\theta(\veczero) = 0$.

Then, for $\Circ^\star \subseteq S_{y,z}$, we define the event $E_{\Circ^\star} := \{\theta \in \Theta_{y,z}: \Gamma(\theta) = \Circ^\star\}$, and will bound from above $\pi_{\gamma,S_{y,z}}^{\omega^{(o)}}(E_{\Circ^\star})$, for every $\Circ^\star$ such that $E_{\Circ^\star}$ is nonempty. We make some more notation: for every such a set $\Circ^\star$, define $O(\Circ^\star)$ (for `outside') as the connected component of $(\Circ^\star)^c$ containing $\partial^\star S_{y,z}$, and define $I(\Circ^\star)$ (for `inside') as $S_{y,z} \setminus (\Circ^\star \cup O(\Circ^\star))$. Then $\Circ^\star$, $I(\Circ^\star)$, and $O(\Circ^\star)$ form a partition of $\overline{S}^\star_{y,z}$. We note that there cannot be a pair of adjacent sites from $I(\Circ^\star)$ and $O(\Circ^\star)$ respectively, since they would then be in the same connected component of $(\Circ^\star)^c$. We also note that for every $\theta \in E_{\Circ^\star}$, $\Circ^\star \subseteq  \Sigma_{\veczero}(\theta) \subseteq \Circ^\star \cup I(\Circ^\star)$ and $K(\theta) \subseteq O(\Circ^\star)$ though the sets need not be equal, since $ \Sigma_{\veczero}(\theta)$ or $K(\theta)$ could contain ``holes'' which are ``filled in'' in $I(\Circ^\star)$ and $O(\Circ^\star)$, respectively.

\begin{figure}[ht]
\centering
\includegraphics[scale = 0.5]{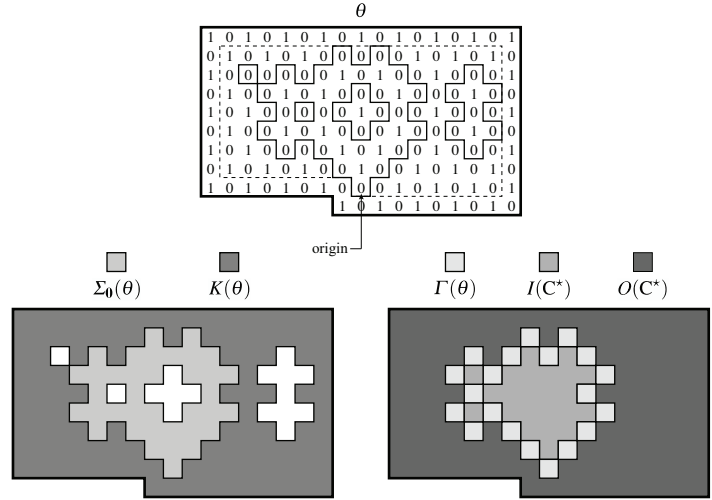}
\caption{A configuration $\theta \in E_{\Circ^\star}$. On the left, the associated sets $\Sigma_{\veczero}(\theta)$ and $K(\theta)$. On the right, the sets $I(\Circ^\star)$ and $O(\Circ^\star)$ for $\Gamma(\theta) = \Circ^\star$.}
\label{hard-core-pic}
\end{figure} 

Choose any set $\Circ^\star$ such that $E_{\Circ^\star} \neq \emptyset$. For each $\theta \in E_{\Circ^\star}$ and $x \in {\Circ^\star}$, using the definition of ${\Circ^\star}$ and the fact that $K(\theta) \subseteq O(\Circ^\star)$, there exists $x_0 \in \{e_1, -e_1, e_2, -e_2\}$ for which $x - x_0 \in O(\Circ^\star)$. Fix an $x_0$ which is associated to at least $|{\Circ^\star}|/4$ of the sites in ${\Circ^\star}$ in this way. Then, we define a function $s: E_{\Circ^\star} \to \{0,1\}^{\overline{S}^\star_{y,z}}$ that, given $\theta \in E_{\Circ^\star}$, defines a new configuration $s(\theta)$ as follows:
\begin{equation}
(s(\theta))(x) = 
\begin{cases}
\theta(x-x_0)	&	\mbox{if $x \in I(\Circ^\star)$},							\\
\theta(x)		&	\mbox{if $x \in O(\Circ^\star)$},							\\
1			&	\mbox{if $x \in {\Circ^\star}$ and $x - x_0 \in O(\Circ^\star)$},	\\
0			&	\mbox{if $x \in {\Circ^\star}$ and $x - x_0 \in I(\Circ^\star)$}.	\\

\end{cases}
\end{equation}

Informally, we move all $1$ symbols inside $I(\Circ^\star)$ in the $x_0$-direction by $1$ unit (even if those symbols were not part of $ \Sigma_{\veczero}(\theta)$), add new $1$ symbols at some sites in ${\Circ^\star}$, and leave everything in $O(\Circ^\star)$ unchanged.

It should be clear that $s(\theta)$ has at least $|{\Circ^\star}|/4$ more $1$ symbols than $\theta$ did. We make the following two claims: $s$ is injective on $E_{\Circ^\star}$, and for every $\theta \in E_{\Circ^\star}$, $s(\theta) \in \Theta_{y,z}$. If these claims are true, then clearly $\pi_{\gamma,S_{y,z}}^{\omega^{(o)}}(s(E_{\Circ^\star})) \geq \gamma^{|{\Circ^\star}|/4} \pi_{\gamma,S_{y,z}}^{\omega^{(o)}}(E_{\Circ^\star})$, implying that:
\begin{equation}
\label{peierlsbound}
\pi_{\gamma,S_{y,z}}^{\omega^{(o)}}(E_{\Circ^\star}) \leq \gamma^{-|{\Circ^\star}|/4}.
\end{equation}

Firstly, we show that $s$ is injective. Suppose that $\theta_1 \neq \theta_2$, for $\theta_1,\theta_2 \in E_{\Circ^\star}$. Then there is a site $x$ at which $\theta_1(x) \neq \theta_2(x)$. If $x \in O(\Circ^\star)$, then $(s(\theta_1))(x) = \theta_1(x) \neq \theta_2(x) = (s(\theta_2))(x)$ and so $s(\theta_1) \neq s(\theta_2)$. If $x \in I(\Circ^\star)$, then $(s(\theta_1))(x + x_0) = \theta_1(x) \neq \theta_2(x) = (s(\theta_2))(x + x_0)$, and again $s(\theta_1) \neq s(\theta_2)$. Finally, we note that $x$ cannot be in ${\Circ^\star}$, since at all sites in ${\Circ^\star}$, both $\theta_1$ and $\theta_2$ must have $0$ symbols. 

Secondly, we show that for any $\theta \in E_{\Circ^\star}$, $s(\theta)$ is feasible. All that must be shown is that $s(\theta)$ does not contain adjacent $1$ symbols. We break $1$ symbols in $s(\theta)$ into three categories:
\begin{enumerate}
\item \emph{shifted}, meaning that the $1$ symbol came from shifting a $1$ symbol at a site in $I(\Circ^\star)$ in the $x_0$-direction,
\item \emph{new}, meaning that the $1$ symbol was placed at a site $x \in {\Circ^\star}$ such that $x - x_0 \in O(\Circ^\star)$, or
\item \emph{untouched}, meaning that the $1$ symbol was at a site in $O(\Circ^\star)$ ($\supseteq \partial^\star S_{y,z}$).
\end{enumerate}

Note that untouched $1$ symbols cannot be adjacent to ${\Circ^\star}$: $\theta$ contains all $0$ symbols on ${\Circ^\star}$, and so since ${\Circ^\star} \subseteq  \Sigma_{\veczero}(\theta)$, a $1$ symbol adjacent to a symbol in ${\Circ^\star}$ would be in $ \Sigma_{\veczero}(\theta)$ as well, a contradiction since $ \Sigma_{\veczero}(\theta) \subseteq {\Circ^\star} \cup I(\Circ^\star)$, and so $ \Sigma_{\veczero}(\theta)$ and $O(\Circ^\star)$ are disjoint.

Clearly shifted $1$ symbols cannot be adjacent to each other, since there were no adjacent $1$ symbols in $\theta$. All new $1$'s were placed at sites in ${\Circ^\star}$, and all sites in ${\Circ^\star}$ are even, so new $1$ symbols can't be adjacent to each other. Untouched $1$'s can't be adjacent for the same reason as shifted $1$'s. We now address the possibility of adjacent $1$ symbols in $s(\theta)$ from different categories. A shifted or new $1$ in $s(\theta)$ is at a site in ${\Circ^\star} \cup I(\Circ^\star)$, and an untouched $1$ can't be adjacent to a site in ${\Circ^\star}$ as explained above, and also cannot be adjacent to a site in $I(\Circ^\star)$ since $I(\Circ^\star)$ and $O(\Circ^\star)$ do not contain adjacent sites. Therefore, shifted or new $1$'s can't be adjacent to untouched $1$'s. The only remaining case which we need to rule out is a new $1$ adjacent to a shifted $1$. Suppose that $(s(\theta))(x)$ is a new $1$ and $(s(\theta))(x')$ is a shifted $1$. Then by definition, $x' - x_0 \in I(\Circ^\star)$ and $x - x_0 \in O(\Circ^\star)$. We know that $I(\Circ^\star)$ and $O(\Circ^\star)$ do not contain adjacent sites, so $x - x_0$ and $x' - x_0$ are not adjacent, implying that $x$ and $x'$ are not adjacent. We've then shown that $s(\theta)$ is feasible and then, since $\partial^\star S_{y,z} \subseteq  O(\Circ^\star)$, $s(\theta) \in \Theta_{y,z}$, completing the proof of (\ref{peierlsbound}). 

Recall that every set $\Circ^\star$ which we are considering is $\star$-connected, occupies only even sites, and contains the origin $\veczero$. Then, given $k \in \N$, it is direct to see that the number of such $\Circ^\star$ with $|\Circ^\star| = k$ is less than or equal to $k \cdot t(k)$, where $t(k)$ denotes the number of \emph{site animals} (see \cite{1-klarner} for the definition) of size $k$ (the first $k$ factor comes from the fact that site animals are defined up to translation, and here given a site animal of size $k$, exactly $k$ translations of it will contain the origin $\veczero$). We know that for every $\epsilon > 0$ there exists $C_\epsilon > 0$ such that $t(k) \leq C_\epsilon(\delta + \epsilon)^k$ for every $k$, where $\delta := \lim_{k \to \infty} \left(t(k)\right)^{1/k} \leq 4.649551$ (see \cite{1-klarner}). 

If $\Sigma_{\veczero}(\theta) \cap \tBdry S_n \neq \emptyset$, then $\Sigma_{\veczero}(\theta)$ has to intersect the left, top or right boundary of $S_n$. W.l.o.g., we may assume that $\Sigma_{\veczero}(\theta)$ intersects the right boundary of $S_n$. Then, every vertical segment in the right half of $S_n$ must intersect $\Sigma_{\veczero}(\theta)$ and, therefore, at least one element of its inner external boundary, namely $\Gamma(\theta)$. Then:
\begin{equation}
\Sigma_{\veczero}(\theta) \cap \tBdry S_n \neq \emptyset \implies |\Gamma(\theta)| \geq n.
\end{equation}

Therefore, taking an arbitrary $\epsilon > 0$, we may bound $\pi_{\gamma,S_{y,z}}^{\omega^{(o)}}(\Sigma_{\veczero}(\theta) \cap \tBdry S_n \neq \emptyset)$ from above:
\begin{equation}
\pi_{\gamma,S_{y,z}}^{\omega^{(o)}}(\Sigma_{\veczero}(\theta) \cap \tBdry S_n \neq \emptyset) \leq \sum_{\Circ^\star: |\Circ^\star| \geq n} \gamma^{-|{\Circ^\star}|/4} \leq \sum_{k = n}^{\infty} kC_\epsilon(\delta+\epsilon)^{k} \cdot \gamma^{-k/4},
\end{equation}
which decays exponentially in $n$ as long as $\gamma > (\delta+\epsilon)^4$, independently of $y$ and $z$. Since $\epsilon$ was arbitrary, $\gamma > \gammatwo > \delta^4$ suffices for justifying (\ref{outofphase}), completing the proof.

\bigskip
\noindent
\textbf{Part II: $\gamma < \gamma_1 := 2.48$.} It is known (see \cite{1-vera}) that when $d = 2$ and $\gamma < 2.48$, $\pi^{\mathrm{HC}}_\gamma$ satisfies exponential SSM. Then, by applying Proposition \ref{SSM_subcritical}, we conclude.
\end{proof}


\section{Poly-time approximation for pressure of $\Z^2$ lattice models}
\label{section9}

By a \emph{poly-time approximation algorithm} to compute a number $r$, we mean an algorithm that, given $N \in \N$, produces an estimate $r_N$ such that $\left|r - r_N\right| < \frac{1}{N}$ and the time to compute $r_N$ is polynomial in $N$.

\begin{theorem}
\label{press-aprox}
Let $\Phi$ be a n.n. interaction for a set of restrictions $\mathcal{E}$ and suppose that $\Omega(\mathcal{E})$ satisfies the square block D-condition. Let $\overline{\omega} \in \Omega(\mathcal{E})$ be a periodic point such that $c_\pi(\nu^{\overline{\omega}}) > 0$. In addition, suppose that there exists $C,\alpha > 0$ such that, for every $y,z \geq \vecone n$:
\begin{equation}
\label{exp-decay}
\left| \pi_{n}(\omega) - \pi_{y,z}(\omega) \right| \leq Ce^{-{\alpha}n} \mbox{ over } \omega \in \mathrm{O}(\overline{\omega}).
\end{equation}

Then:
\begin{equation}
\Press(\Phi) = \frac{1}{\left|\mathrm{O}(\overline{\omega})\right|}\sum_{\omega \in \mathrm{O}(\overline{\omega})}{\hat{I}_{\pi}(\omega) + A_\Phi(\omega)},
\end{equation}
and there is a poly-time approximation algorithm to compute $\Press(\Phi)$, when $d = 2$.
\end{theorem}

\begin{proof}
Notice that $\supp(\nu^{\overline{\omega}}) = \mathrm{O}(\overline{\omega}) \subseteq \Omega(\mathcal{E})$, since $\Omega(\mathcal{E})$ is shift-invariant and $\overline{\omega} \in \Omega(\mathcal{E})$. Now, since $\left| \pi_{n}(\omega) - \pi_{y,z}(\omega) \right| \leq Ce^{-{\alpha}n} \mbox{ over } \omega \in \supp(\nu^{\overline{\omega}})$, we can easily conclude that $\lim_{y,z \rightarrow \infty} \pi_{y,z}(\omega) = \hat{\pi}(\omega)$ uniformly over $\omega \in\supp(\nu^{\overline{\omega}})$. This, combined with $\Omega(\mathcal{E})$ satisfying the square block D-condition and $\mathrm{c}_\pi(\nu^{\overline{\omega}}) > 0$, gives us
\begin{equation}
\Press(\Phi) = \int{\left(\hat{I}_{\pi} + A_\Phi\right)}d\nu^{\overline{\omega}} = \frac{1}{|\supp(\nu^{\overline{\omega}})|}\sum_{\omega \in \supp(\nu^{\overline{\omega}})}{\hat{I}_{\pi}(\omega) + A_\Phi(\omega)},
\end{equation}
thanks to Theorem \ref{press-rep}.

For the algorithm, it suffices to show that there is a poly-time algorithm to compute $\hat{\pi}(\omega)$, for any $\omega \in \mathrm{O}(\overline{\omega})$.

By Equation \ref{exp-decay}, there exist $C, \alpha > 0$ such that $\left|\pi_n(\omega) - \hat{\pi}(\omega)\right| < Ce^{-\alpha n}$. Since $|\partial S_n|$ is linear in $n$ when $d = 2$, by a modified transfer matrix approach (see \cite[Lemma 4.8]{2-marcus}), we can compute $\pi_n(\omega)$ in exponential time $Ke^{\rho n}$ for some $K, \rho > 0$. Combining the exponential time to compute $\pi_n(\omega)$ for the exponential decay of $\left|\pi_n(\omega) - \hat{\pi}(\omega)\right|$, we get a poly-time algorithm to compute $\Press(\Phi)$: namely, given $N \in \N$, let $n$ be the smallest integer such that $Ce^{-\alpha (n+1)} < \frac{1}{N}$. Then $\pi_{n+1}(\omega)$ is within $\frac{1}{N}$ of $\hat{\pi}(\omega)$ and since $\frac{1}{N} \leq Ce^{-\alpha n}$, the time to compute $\pi_{n+1}(\omega)$ is at most:
\begin{equation}
Ke^{\rho (n+1)} = (Ke^\rho C^{\rho/\alpha})\frac{1}{(Ce^{-\alpha n})^{\rho/\alpha}} \le (Ke^\rho C^{\rho/\alpha})N^{\rho/\alpha},
\end{equation}
which is a polynomial in $N$.
\end{proof}

\begin{corollary}
\label{potts-wr-hard}
The following holds:
\begin{enumerate}
\item For the $\Z^2$ Potts model with $q$ types and inverse temperature $\beta > 0$:
\begin{equation}
\label{pres_Potts}
\Press(\Phi_\beta) = \hat{I}^\beta_\pi(\omega_q) + 2\beta.
\end{equation}

\item For the $\Z^2$ Widom-Rowlinson model with $q$ types and activity $\lambda \in (0,\lambda_1(q)) \cup (\lambda_2(q),\infty)$:
\begin{equation}
\label{pres_WR}
\Press(\Phi_{\lambda}) = \hat{I}^\lambda_\pi(\omega_q) + \log \lambda,
\end{equation}
where $\lambda_1(q) := \frac{1}{q}\left(\frac{p_c}{1-p_c}\right)$ and $\lambda_2(q) := q^3\left(\frac{p_c}{1-p_c}\right)$.

\item For the $\Z^2$ hard-core model with activity $\gamma \in (0,\gamma_1) \cup (\gamma_2,\infty)$:
\begin{equation}
\label{pres_hardcore}
\Press(\Phi_{\gamma}) = \frac{1}{2}\hat{I}^\gamma_\pi(\omega^{(o)}) + \frac{1}{2} \log \gamma,
\end{equation}
where $\gamma_1 = 2.48$ and $\gamma_2 = \gammatwo$.
\end{enumerate}

Moreover, for the three models in the corresponding regions (except in the case when $\beta = \beta_{\rm c}(q)$ in the Potts model), the pressure can be approximated in poly-time, where the polynomial involved depends on the parameters of the models.
\end{corollary}

\begin{proof}
The representation of the pressure given in the previous statement for the $\Z^2$ Potts model with $q$ types and inverse temperature $\beta \neq \beta_{\rm c}(q)$, the $\Z^2$ Widom-Rowlinson model with $q$ types and activity $\lambda \in (0,\lambda_1(q)) \cup (\lambda_2(q),\infty)$ and the $\Z^2$ hard-core model with activity $\gamma \in (0,\gamma_1) \cup (\gamma_2,\infty)$, is a direct consequence of Theorem \ref{press-aprox}, by virtue of the following facts:
\begin{itemize}
\item Recall that the corresponding n.n. SFT $\Omega(\mathcal{E})$ for the Potts, Widom-Rowlinson and hard-core model has a safe symbol, respectively, so $\Omega(\mathcal{E})$ satisfies the square block D-condition and $c_\pi(\nu) > 0$, for any shift-invariant $\nu$ with $\supp(\nu) \subseteq \Omega(\mathcal{E})$, in each case.
\item If we consider the delta-measure $\nu = \nu^{\omega_q} = \delta_{\omega_q}$, both in the Potts and Widom-Row\-lin\-son cases (in a slight abuse of notation, since the Potts and Widom-Row\-lin\-son $\sigma$-algebras are defined in different alphabets), or the measure $\nu = \nu^{\omega^{(o)}} = \frac{1}{2}\delta_{\omega^{(e)}} + \frac{1}{2} \delta_{\omega^{(o)}}$ in the hard-core case, we have that in all three models, for the range of parameters specified, except for when $\beta = \beta_{\rm c}(q)$ in the Potts model, there exists $C,\alpha > 0$ such that, for every $y,z \geq \vecone n$:
\begin{equation}
\left| \pi_{n}(\omega) - \pi_{y,z}(\omega) \right| \leq Ce^{-{\alpha}n}, \mbox{ over } \omega \in \supp(\nu),
\end{equation}
thanks to Theorem \ref{potts-decay}, Theorem \ref{widom-decay} and Theorem \ref{hard-decay}, respectively. (Notice that $\hat{I}^\gamma_\pi(\omega^{(e)}) = A_\Phi(\omega^{(e)}) = 0$.)
\end{itemize}

This proves (\ref{pres_Potts}), (\ref{pres_WR}) and (\ref{pres_hardcore}), except in the Potts case when $\beta = \beta_{\rm c}$. To establish this case, first note that it is easy to prove that $\Press(\Phi_\beta)$ is continuous with respect to $\beta$. Second, if $\beta_1 \leq \beta_2$, then $\pi^{\beta_1}_{n}(\omega_q) \leq \pi^{\beta_2}_{n}(\omega_q)$. This follows by the Edwards-Sokal coupling (see Theorem \ref{edward}) and the comparison inequalities for the bond random-cluster model \cite[Theorem 4.1]{1-aizenman}.

As an exercise in analysis, it is not difficult to prove that if $a_{m,n} \geq 0$, and each $a_{m+1,n} \leq a_{m,n}$ and $a_{m,n+1} \leq a_{m,n}$, then $\lim_m \lim_n a_{m,n} = \lim_n \lim_m a_{m,n} = a$, for some $a \geq 0$.

Now, consider the sequence $a_{m,n} := \pi^{\beta_{\rm c}(q) + \frac{1}{m}}_{n}(\omega_q)$. By stochastic dominance (see Proposition \ref{Potts_FKG}), $a_{m,n}$ is decreasing in $n$. By the previous discussion (Edwards-Sokal coupling), it is also decreasing in $m$. Therefore, and since $a_{m,n} \geq 0$, we conclude that $\lim_m \lim_n a_{m,n} = \lim_n \lim_m a_{m,n} = a$, for some $a$.

Then, we have that:
\begin{align}
\Press(\Phi_{\beta_{\rm c}(q)})	&	=	\lim_m \Press(\Phi_{\beta_{\rm c}(q) + \frac{1}{m}})	\\
					&	=	\lim_m -\log \lim_n \pi^{\beta_{\rm c}(q) + \frac{1}{m}}_n(\omega_q) + 2\left(\beta_{\rm c}(q) + \frac{1}{m}\right)	\\
					&	=	-\log \lim_m \lim_n \pi^{\beta_{\rm c}(q) + \frac{1}{m}}_n(\omega_q) + 2\beta_{\rm c}(q)	\\
					&	=	-\log \lim_n \lim_m \pi^{\beta_{\rm c}(q) + \frac{1}{m}}_n(\omega_q) + 2\beta_{\rm c}(q)	\\
					&	=	-\log \lim_n \pi^{\beta_{\rm c}(q)}_n(\omega_q) + 2\beta_{\rm c}(q)	\\
					&	=	\hat{I}^{\beta_{\rm c}(q)}_\pi(\omega_q) + 2\beta_{\rm c}(q).
\end{align}

(To prove that $\lim_m \pi^{\beta_{\rm c}(q) + \frac{1}{m}}_n(\omega_q) = \pi^{\beta_{\rm c}(q)}_n(\omega_q)$ is straightforward.)

Finally, the algorithmic implications are also a direct application of Theorem \ref{press-aprox}.
\end{proof}

\begin{remark}
The algorithm given in Theorem \ref{press-aprox} seems to require explicit bounds on the constants $C$ and $\alpha$, so that given $N \in \N$, we can find an explicit $n$ such that $Ce^{-\alpha (n+1)} < \frac{1}{N}$. Without such bounds, while there exists a poly-time algorithm, we do not always know how to exhibit an explicit algorithm. However, for all three models, for regions sufficiently deep within the supercritical region (i.e. $\beta$, $\lambda$ or $\gamma$ sufficiently large), one can find crude, but adequate, estimates on $C$ and $\alpha$ and thus can exhibit a poly-time algorithm. This is the case for the hard-core model, where our proof does allow an explicit estimate of the constants for any $\gamma > \gammatwo$. On the other hand, in the regions specified in Corollary \ref{potts-wr-hard} within the subcritical region, all three models satisfy exponential SSM and then using \cite[Corollary 4.7]{2-marcus}, one can, in principle, exhibit a poly-time algorithm (even without estimates on $C$ and $\alpha$).
\end{remark}


\section*{Acknowledgements}

We thank Nishant Chandgotia and Andrew Rechnitzer for helpful discussions.

\printbibliography

\end{document}